\documentclass[10pt,oneside,english,reqno,twoside]{amsart}
\usepackage{lmodern}

\usepackage[T1]{fontenc}
\usepackage[latin9]{luainputenc}
\usepackage{geometry}
\usepackage{amstext}
\usepackage{amsthm}
\usepackage{enumitem}
\usepackage{tikz}
\usetikzlibrary{patterns}
\usepackage{amssymb}
\usepackage{xcolor}
\definecolor{weblmcolor}{cmyk}{0.86,0.23,0.44,0.15}
\definecolor{dmagenta}{rgb}{.4,.1,.4}
\definecolor{dblue}{rgb}{.0,.0,.5}
\definecolor{mblue}{rgb}{.0,.4,.7}
\definecolor{ddblue}{rgb}{.0,.0,.4}
\definecolor{dred}{rgb}{.9,.0,.0}
\definecolor{dgreen}{rgb}{.0,.5,.0}
\definecolor{Eeom}{rgb}{.0,.0,.5}
\definecolor{dbrown}{rgb}{.5,.2,.0}
\usepackage[pdfusetitle,colorlinks,linkcolor=weblmcolor, citecolor=brown, urlcolor=weblmcolor]{hyperref}

\usepackage[capitalise]{cleveref}
\makeatletter
\numberwithin{equation}{section}
\numberwithin{figure}{section}
\theoremstyle{plain}
\newtheorem{thm}{\protect\theoremname}[section]
  \theoremstyle{definition}
  \newtheorem{defn}[thm]{\protect\definitionname}
  \theoremstyle{plain}
  \newtheorem{lem}[thm]{\protect\lemmaname}
  \theoremstyle{remark}
  \newtheorem{rem}[thm]{\protect\remarkname}

\usepackage{versions}

\DeclareMathOperator*{\esssup}{ess\,sup}

\DeclareMathOperator*{\esssinf}{ess\,inf}

\DeclareMathOperator*{\esssliminf}{ess\,lim\,inf}

\def\Xint#1{\mathchoice
{\XXint\displaystyle\textstyle{#1}}%
{\XXint\textstyle\scriptstyle{#1}}%
{\XXint\scriptstyle\scriptscriptstyle{#1}}%
{\XXint\scriptscriptstyle\scriptscriptstyle{#1}}%
\!\int}
\def\XXint#1#2#3{{\setbox0=\hbox{$#1{#2#3}{\int}$ }
\vcenter{\hbox{$#2#3$ }}\kern-.6\wd0}}

\def\dashint{\Xint-}

\usepackage{hyperref}
\newcommand{\Addresses}{{
\bigskip
\footnotesize

\textsc{Abhrojyoti Sen,  Goethe-Universit\"{a}t Frankfurt, Institut f\"{u}r Mathematik, Robert-Mayer-Str. 10, D-60629 Frankfurt, Germany}\par\nopagebreak
\textit{E-mail address}: \href{mailto:sen@math.uni-frankfurt.de}{sen@math.uni-frankfurt.de}
}}

\makeatother

\usepackage{babel}
  \providecommand{\definitionname}{Definition}
  \providecommand{\lemmaname}{Lemma}
  \providecommand{\remarkname}{Remark}
\providecommand{\theoremname}{Theorem}

\begin{document}
\global\long\def\d{\,d}
\global\long\def\tr{\mathrm{tr}}
\global\long\def\supp{\operatorname{spt}}
\global\long\def\div{\operatorname{div}}
\global\long\def\osc{\operatorname{osc}}
\global\long\def\essup{\esssup}
\global\long\def\aint{\dashint}
\global\long\def\essinf{\esssinf}
\global\long\def\essliminf{\esssliminf}

\excludeversion{old}

\excludeversion{note}

\excludeversion{note2}


\title[1D pressureless gas dynamics systems in a strip]{1D pressureless gas dynamics systems in a strip}

\author{Abhrojyoti Sen}
\begin{abstract}
We construct explicit measure-valued solutions to the one-dimensional pressureless gas dynamics system in a strip-like domain by introducing a new boundary potential. The constructed solutions satisfy an entropy condition, and depending on the boundary data and the behavior of the potentials, mass accumulation can occur at the boundaries. The approach relies on a systematic treatment of boundary potentials and their interactions with the initial data, providing a more precise understanding of the formation and propagation of singularities in measure-valued solutions.
\end{abstract}

\maketitle
\tableofcontents
\section{Introduction and main result}
\subsection{State of the art}
In this article, we investigate the solvability of the initial-boundary value problem for the system of pressureless gas dynamics described by
\begin{equation} \begin{aligned} &\partial_t \rho + \partial_x (\rho u) = 0,\\ &\partial_t (\rho u) + \partial_x (\rho u^2) = 0, \end{aligned} \label{e1.1} \end{equation}
in one spatial dimension, where $\rho=\rho(x, t)\geq 0$ denotes the density and $u=u(x,t)$ represents the velocity field. The two equations in system \eqref{e1.1} models the conservation of mass and momentum, respectively. We consider \eqref{e1.1} subject to the initial conditions
\begin{equation}\label{e1.2}
\rho(x,0) = \rho_0(x), \quad u(x,0) = u_0(x), \quad x \in [0,1],  
\end{equation}
where $\rho_0 : [0, 1]\to \mathbb{R}^+$ and $u_0:[0,1]\to \mathbb{R}$ are given functions. The focus of our analysis is to determine the admissibility of boundary data
\begin{align}
&u(0,t) = u_{b_l}(t),\,\,\, \rho(0,t) u(0,t) = \rho_{b_l}(t) u_{b_l}(t),\label{e1.3}\\
&u(1,t) = u_{b_r}(t),\,\,\, \rho(1,t) u(1,t) = \rho_{b_r}(t) u_{b_r}(t),  \quad t > 0,\label{e1.4}    
\end{align} 
where $u_{b_l}, u_{b_r}: [0, \infty)\to \mathbb{R}$ and $\rho_{b_l}, \rho_{b_r}: [0, \infty)\to \mathbb{R}^+$ are prescribed functions. Our analysis aims to identify the conditions under which the boundary data in \eqref{e1.3} can be consistently prescribed and to rigorously determine the sense in which solutions to \eqref{e1.1} exist. In particular, we examine whether the system admits weak or measure-valued solutions that satisfy the entropy condition and respect the prescribed initial and boundary data. For scalar conservation laws, this type of problem in a strip $[0,1] \times (0, \infty)$ has been studied by Frankowska \cite{F2010}, and Joseph and Gowda \cite{JG92}. For applications in highway traffic modeling, see the work of Strub and Bayen \cite{SB2006}. Moreover, \cite{L2018} and the references therein provide a broader overview of the initial boundary value problems for related hyperbolic systems.

We consider the spatial domain $[0, 1]$ for convenience; however, any other finite interval can be addressed by applying an appropriate affine rescaling transformation. Specifically, for an interval $[a, b]$ the rescaling can be achieved via the mapping $x \mapsto \frac{x-a}{b-a}.$ The initial datum $(\rho_0, u_0)$ are assumed to be locally bounded measurable functions with $\rho_0>0.$ The boundary datum are prescribed to impose the inflow conditions. Specifically, the functions $u_{b_l}(t)$ and $u_{b_r}(t)$ are bounded measurable, with $u_{b_l}(t)>0$ and $u_{b_r}(t)<0$ for all $t>0.$ Furthermore, the associated density functions $\rho_{b_l}(t), \rho_{b_r}(t)$ are assumed to be positive, locally bounded, and measurable. These assumptions ensure the existence of the inflow boundary conditions for the system. They also allow for the application of weak or measure-valued solution frameworks, which will be employed to analyze the solvability of the problem under consideration.

The initial value problem \eqref{e1.1}-\eqref{e1.2} has been extensively studied over the past few decades. It is now well established that one of the primary challenges lies in the fact that $\rho$ is not generally a function, but a Radon measure. Consequently, the natural framework for seeking weak solutions to \eqref{e1.1}-\eqref{e1.2} is the space of Radon measures.

Several notions of weak solutions have been introduced in the literature, including measure-valued solutions \cite{Bouchut94}, duality solutions, and solutions derived via the vanishing viscosity method \cite{BouchutJames99, BouchutJames98, Boudin00}. The global existence of weak solutions using mass and momentum potentials was demonstrated in \cite{BrenierGrenier98, d3}, while explicit formulas for solutions were obtained using generalized potentials and variational principles in \cite{ERykovSinai96, Wang01, WHD97}.

 A novel approach to the global existence of weak solutions for the 1D pressureless gas dynamics equations was introduced by Natile and Savar\'{e} \cite{NS2009}, who constructed sticky particle solutions using a metric projection onto the cone of monotone maps. Subsequently, Cavalletti et al. \cite{cavalletti} provided a more direct proof by employing the concept of differentiability of metric projections, as developed by Haraux.

In another line of work, Nguyen and Tudorascu \cite{nguyen, NT08} established a general global existence result for \eqref{e1.1}-\eqref{e1.2}, with or without viscosity, by constructing entropy solutions for suitable scalar conservation laws. They further proved the uniqueness of solutions using the contraction principle in the Wasserstein metric. Additional uniqueness results were provided by Wang and Ding \cite{WangDing97} and Huang and Wang \cite{Wang01}, who used generalized characteristics introduced by Dafermos \cite{dafermos1}.

In contrast, Bressan and Nguyen \cite{BN2014} demonstrated non-uniqueness and non-existence of solutions in the multidimensional case by constructing specific initial data.

Regarding the numerical methods for the pressureless gas dynamics model, we refer to \cite{mathiaud, bouchut-jin} for details. However, the initial-boundary value problem for \eqref{e1.1} with \eqref{e1.2}-\eqref{e1.3} was first investigated in \cite{NOSS22}, making a significant advancement in the study of this class of problems.
\subsection{Main result} In this subsection, we present our main result. Building upon the existing literature, we investigate the global existence of an entropy solution to the system \eqref{e1.1}-\eqref{e1.4}. As noted earlier, Neumann et al. \cite{NOSS22} constructed measure-valued solutions for the initial-boundary value problem \eqref{e1.1}-\eqref{e1.3}, utilizing the framework of generalized potentials. For the sake of completeness, we provide a short overview of the method and the associated solution concept.

Rykov et al. \cite{ERykovSinai96} introduced a generalized variational principle for pressureless gas dynamics, extending the classical variational principle of Lax and Oleinik for scalar conservation laws, particularly the Burgers equation. Subsequent works by Huang and Wang \cite{Wang01} and Ding et al. \cite{WHD97} extended the method of generalized potentials to cases where the initial datum $u_0$ is not continuous and $\rho_0 \geq 0$ is a Radon measure, respectively. In this setting, the solution concept is the following: we show that $(\rho,u)$ is actually a weak solution to the system \eqref{e1.1}. The construction begins with locally bounded measurable functions $m(x,t)$ and $u(x,t),$ where $m(x,t)$ is of locally bounded variation ($BV$) in $x$ for almost every $t.$ This function $m$ defines a Lebesgue-Stieltjes measure $dm$ whose derivative in the distributional sense corresponds to a Radon measure $\rho=m_x.$ These two representations are same through the identification
\begin{align*}
    -\left\langle m, \varphi_x \right\rangle=-\int_{-\infty}^{\infty}\varphi_x m\, dx=\int_{-\infty}^{\infty}\varphi m_x \,dx=\int_{-\infty}^{\infty}\varphi\, dm=\left\langle v, \varphi \right\rangle \,\,\,\, \text{for all}\,\,\,\, \varphi\in C^{\infty}_c(\mathbb{R}).
\end{align*}
Furthermore, similar identification allows us to define
\begin{align*}
    \left\langle vu, \varphi\right\rangle=\int_{-\infty}^{\infty}\varphi u \, dm.
\end{align*}
These identifications lead to the notion of generalized solution to \eqref{e1.1}. The first equation of \eqref{e1.1} can then be expressed in the distributional sense as:
\begin{align*}
    0=\left\langle v, \varphi_t\right\rangle+ \left\langle vu ,\varphi_x\right\rangle =-\iint \varphi_{xt} m \,dx\,dt+\iint \varphi_x u \,dm\, dt.
\end{align*}
Similarly, the second equation of \eqref{e1.1} takes the form:
\begin{align*}
  0=\left\langle vu, \varphi_t\right\rangle+ \left\langle vu^2 ,\varphi_x\right\rangle =\iint u \varphi_t \,dm\,dt+\iint u^2 \varphi_x dm\, dt.
\end{align*}
Therefore, the weak formulation to the system \eqref{e1.1} is the following:
\begin{defn}\label{intro-defn-weak formulation}
    The pair $(\rho, u)$ is said to be a generalized solution to the system \eqref{e1.1} if the following integral identities
\begin{align}
    &\iint \varphi_t m\,dx\, dt-\iint \varphi u \,dm\,dt=0,\label{weak formulation}\\
    &\iint \varphi_t u + \varphi_x u^2 \, dm\,dt=0,\label{wf2}
\end{align} hold for all test functions $\varphi \in C^{\infty}_c((0,1)\times\mathbb{R}^+),$ where the distributional derivative $m_x$ defines the Radon measure $\rho.$ 
\end{defn}
The construction of the generalized solution $(m, u)$ proceeds in two stages. First, we introduce generalized potentials $F(y, x,t), G_{b_l}(\tau, x, t), G_{b_r}(\xi, x, t),$ to construct $u$ and $m.$ Subsequently, we define the momentum and energy potentials $q(x,t)$ and $E(x,t),$ respectively, alongside an auxiliary functional $H(\cdot, x,t).$ Moreover, by establishing relations between the measures $dq,$ $dE,$ and  $dm,$ we verify that $q, E, m$ satisfy Definition \ref{intro-defn-weak formulation}.

Now we state our main result below.
\begin{thm}\label{TH1.3}
Let $\rho_0(x), \rho_{b_l}(t)>0, \rho_{b_r}(t)>0,$ and $u_0(x), u_{b_l}(t)>0, u_{b_r}(t)<0,$ are locally bounded measurable functions, then the pair $(m,u)$ given by Definition \ref{d3}-\ref{d2} is a global weak entropy solution to the system \eqref{e1.1}-\eqref{e1.4} in the sense of Definition \ref{intro-defn-weak formulation} and \ref{def:sol}.
\end{thm}
It should be noted that in the above theorem, the initial condition holds in its classical sense. Turning to the boundary conditions \eqref{e1.3}-\eqref{e1.4}, we recall that, as established in the theory of conservation laws \cite{bar,JosephSahoo12,Le1}, boundary data cannot be arbitrarily prescribed.  One cannot arbitrarily prescribe boundary data, because a priori there is no control of the sign of $u(0+,t)$ and $u(1-, t)$, except in the case when $u_0$ is positive and hence $u(x,t) > 0$ everywhere. At this point,  we recall that $u_{b_l}$ was assumed to be positive, whereas $u_{b_r}$ is assumed to be negative.\\
We shall show the following: If $t>0$ is a Lebesgue point of $u_{b_l}$ and $\rho_{b_l}$ and $\min\left\{F(0,t), G_{b_r}(0,t)\right\} > G_{b_l}(0,t)$, then $\lim_{x\to 0+} u(x,t) = u_{b_l}(t)$. If in addition $u_{b_l}$ is continuously differentiable and $\rho_{b_l}$ is locally Lipschitz continuous, then $\lim_{x\to 0+} \rho(x,t)u(x,t) = \rho_{b_l}(t)u_{b_l}(t)$. On the other hand, if  $\min\left\{F(0,t), G_{b_r}(0,t)\right\} \leq G_{b_l}(0,t)$ then $u(0+,t) < 0$ and the boundary condition \eqref{e1.3} cannot be fulfilled. Rather, it may happen that mass accumulates at the boundary in the form of $\delta\cdot (m(0+,t) - m(0,t))$. We show similar assertions for the right boundary also. We prove that the boundary condition holds in its classical sense when $\min\left\{F(1, t), G_{b_l}(1, t)\right\}>G_{b_r}(1, t),$ otherwise mass can accumulate at the right boundary. This phenomenon is exhibited through some specific example with Riemann type data in Section \ref{sec:5}.  For further aspects of boundary conditions for systems involving measure solution, see \cite{ma2,NNOS17}.
\subsection{Extensions and open problems} In this subsection, we mention some possible extensions and open problems related to the gas dynamics model in the initial boundary setting. 

\noindent $\bullet$ Recently, in \cite{Sen2024}, Sen and Sen investigated the pure initial value problem for a nonhomogeneous gas dynamics model where the nonhomogeneous term depends not only on the solutions $(\rho, u),$ but also explicitly on $x$ and $t.$ Specifically, the system is given by:
    \begin{align*}
        &\partial_t \rho + \partial_x (\rho u) = 0,\\ &\partial_t (\rho u) + \partial_x (\rho u^2) = \frac{1}{t+\alpha}\left(\frac{x}{t+\alpha}-u\right)\rho, \,\,\, \alpha \in \mathbb{R}^+.
    \end{align*}
A natural extension of this work would involve constructing a measure-valued solution for the initial-boundary value problem associated with the above system. In \cite{Sen2024}, an initial potential 
 \[F(x,t):= \min_{y\in \mathbb{R}} F(y, x, t)=\int_0^y\left(\left(\frac{\eta}{2\alpha}+\frac{u_0(\eta)}{2}\right)(t+\alpha)+\left(\frac{\eta}{2}-\frac{u_0(\eta)\alpha}{2}\right)\frac{\alpha}{(t+\alpha)}-x\right) \rho_0(\eta)d\eta\]
 was introduced to analyze the initial value problem. To study the corresponding initial-boundary value problem, it would be necessary to introduce appropriate boundary potentials and define suitable characteristic triangles. This undertaking appears to be both intricate and mathematically rigorous, requiring careful analysis to address the added difficulties introduced by the nonhomogeneous terms.\\

\noindent $\bullet$ Next, we consider the study of a gas dynamics system with a discontinuous (in space) right-hand side. An example of such a system is given by:
 \begin{align*}
    &\partial_t \rho + \partial_x (\rho u) = 0,\\ &\partial_t (\rho u) + \partial_x (\rho u^2) = H(x)\rho
 \end{align*}
 with the initial data $(\rho(x, 0), u(x, 0))=(\rho_0(x), u_0(x)),$ where a discontinuous function $H(x)$ is defined as:
 $$H(x)=\begin{cases}
     \alpha, \,\,\,\, &\text{if}\,\,\, x>0\\
     \beta \,\,\,\, &\text{if}\,\,\, x<0,
 \end{cases}$$
 for $\alpha, \beta \in \mathbb{R}, \alpha\neq \beta.$ In other words, the above problem can be interpreted as a gas dynamics system with a discontinuous flux:
 \begin{align*}
    &\partial_t \rho + \partial_x \left(\rho f(x, t, v)\right) = 0,\\ &\partial_t (\rho v) + \partial_x \left(\rho v f(x, t, v)\right) =0,
 \end{align*}
 where
 \begin{align*}
     f(x, t, v)=\begin{cases}
         v+\alpha t,\,\,\, x>0\\
         v+\beta t,\,\,\, x<0,
     \end{cases}
     \,\,\,\,\,\,\,\,\,\,\text{and}\,\,\,\,\,\,\,\,\,\,v=\begin{cases} 
     u-\alpha t,\,\,\, x>0\\
     u-\beta t,\,\,\, x<0.
     \end{cases}
 \end{align*}
 Given that \cite{NOSS22} provides techniques to address the left boundary and the current work develops methods to handle the right boundary, it remains an open question whether these ideas can be combined to analyze the problem described above. It is important to note that this is an initial value problem, with no data prescribed on the interface $x=0.$ To address this, one must carefully study the behavior of the characteristic lines (or curves) passing through the interface. The discontinuity on the right-hand side introduces additional complexities, particularly in understanding how the interface influences the structure of solutions. Moreover, these aspects require rigorous analysis to ensure a well-defined solution framework.

\subsection{Organization of the article} The paper is organized as follows: In Section \ref{sec:2}, we construct solutions, i.e., $(m, u),$ based on generalized potentials. Section \ref{sec:3} is dedicated to demonstrating that the constructed solution satisfies the weak formulation in the sense of Definition \ref{intro-defn-weak formulation}, as well as discussing the initial-boundary conditions and the entropy admissibility of solution. In what follows, we prove Theorem \ref{TH1.3}. Finally, Section \ref{sec:5} presents specific examples that validate and illustrate the theoretical results developed in this work.
\section{Construction of solution}\label{sec:2}
This section defines the potentials for the initial and boundary manifolds. The smallest potential at $(x,t),$ determines characteristic triangles that specify the regions influencing the solution. These non-intersecting triangles form a continuous curve, with its time derivative matching the velocity from momentum influx. For points $(x,t)$ in the domain $\Omega=[0,1]\times[0,\infty[$, we introduce the initial and boundary functionals as
\begin{align*}
&F(y,x,t)=\int_{0}^{y}(tu_0(\eta)+\eta-x)\rho_0(\eta)d\eta\,,\\
&G_{b_l}(\tau, x, t)=\int_{0}^{\tau} [x-u_{b_l}(\eta)(t-\eta)]\rho_{b_l}(\eta)u_{b_l}(\eta)d\eta\,, \\
&G_{b_r}(\xi, x, t)=\int_{0}^{\xi} [x-1-u_{b_r}(\eta)(t-\eta)]\rho_{b_r}(\eta)u_{b_r}(\eta)d\eta+F(1,x,t)\,. 
\end{align*}
Given $(x, t),$ we denote $y_{*}(x,t)$ and $y^* (x,t)$ for the left most and right most points respectively in the interval [0,1], such that 
\begin{equation}\label{InitPot}
F(x,t):=\min_{0\leq y\leq 1}F(y,x,t)= F(y_{*}(x,t), x,t)= F(y^{*}(x,t), x,t)\,.
\end{equation}
$F(x,t)$ is called the initial potential.

Similarly, given $(x,t)$, we denote by $\tau^{*}(x,t)$ and $\tau_{*}(x,t)$ the upper most and lower most points on the $t$-axis respectively such that
\begin{equation}\label{leftBdpot}
G_{b_l}(x,t):=\min_{\tau \geq 0}G_{b_l}(\tau, x, t)= G(\tau^{*}(x,t), x, t)=G(\tau_{*}(x,t), x, t)\,.
\end{equation}
We call $G_{b_l}(x,t)$ the left boundary potential.
And finally, for the right boundary, we call these points $\xi^{*}(x,t)$ and $\xi_{*}(x,t)$ and define
\begin{equation}\label{rightBdpot}
G_{b_r}(x,t):=\min_{\xi \geq 0}G_{b_r}(\xi, x, t)=G(\xi^{*}(x,t), x, t)=G(\xi_{*}(x,t), x, t)\,,
\end{equation}
where $G_{b_r}(x,t)$ is the right boundary potential.
Note that $\tau^{*}(x,t)\leq t $ and $\xi^{*}(x,t)\leq t$ by straightforward arguments using the signs of the integrands.

All these minima, and thus the potentials, exist as real numbers, because of the signs of $u_{b_l}$ and $u_{b_r}$. Observe that for a fixed $t>0$ the functions $F(x,t)$ and $G_{b_r}(x,t)$ are monotonically decreasing in $x$ while $G_{b_l}(x,t)$ is monotonically increasing.
\begin{rem}
The key difference between the left boundary potential $G_{b_l}(x, t)$, as introduced in \cite{NOSS22}, and the newly defined right boundary potential $G_{b_r}(x, t)$ is the explicit dependence of $G_{b_r}(x, t)$ on the initial potential $F(x,t)$. More precisely, while $G_{b_l}$ is constructed solely based on boundary data at $x=0,$ the formulation of $G_{b_r}$ involves contributions from both the boundary data at $x=1$ and the initial potential to accurately capture the mass and its transportation throughout the domain.
\end{rem}
Now we collect some properties of the minimizers.
\begin{lem}\label{lnew}
With our assumptions on the initial and boundary data we have
\begin{enumerate}
\item For fixed $x,$ $\tau_{*}(x,t)$ and $\tau^{*}(x,t)$ are monotonically increasing in $t$ and for fixed $t$ monotonically decreasing in $x$. Moreover, we have for $t_1<t_2$ that $\tau^{*}(x,t_1)\leq\tau_{*}(x,t_2)$ and for $x_1<x_2$ that $\tau_{*}(x_1,t)\geq\tau^{*}(x_2,t)$.
\item For fixed $x,$ $\xi_{*}(x,t)$ and $\xi^{*}(x,t)$ are monotonically increasing in $t$ and for fixed $t$ monotonically increasing in $x$.  For $t_1<t_2$ that $\xi^{*}(x,t_1)\leq\xi_{*}(x,t_2)$ and for $x_1<x_2$ that $\xi^{*}(x_1,t)\leq\xi_{*}(x_2,t)$. 
\item For fixed $t,$ $y_{*}(x,t)$ and $y^{*}(x,t)$ are monotonically increasing in $x$ and for $x_1<x_2$ we have $y^{*}(x_1,t)\leq y_{*}(x_2,t)$.
\item $\tau_{*}(0,t)=\tau^{*}(0,t)=t,$ $\xi_{*}(1,t)=\xi^{*}(1,t)=t$ and $y_{*}(x,0)=y^{*}(x,0)=x$.
\item $y_{*}(x,t)$, $\tau_{*}(x,t)$ and $\xi_{*}(x,t)$ are lower semicontinuous while $y^{*}(x,t)$, $\tau^{*}(x,t)$ and $\xi^{*}(x,t)$ are upper semicontinuous.
\end{enumerate}
\end{lem}

\begin{proof}
A similar result can be found in \cite{NOSS22}, and this is in turn based on the results in \cite{WHD97}. For completeness, we give the proof of (2) and the part of (4) which is new for the domain $\Omega$.\\
To prove (2) let $x,\,t_1,\,t_2>0$ be arbitrary but fixed, and $\xi_1$ a minimizer of $G_{b_r}(\xi,x,t_1)$ and $\xi_2$ one of $G_{b_r}(\xi,x,t_2)$. Now we have
\[
0\leq G_{b_r}(\xi_2,x,t_1)-G_{b_r}(\xi_1,x,t_1)\,,\ \text{and } 0\leq G_{b_r}(\xi_1,x,t_2)-G_{b_r}(\xi_2,x,t_2)\,.
\]
Adding the two inequalities and using the definition of $G_{b_r}$ results in 
\[
0 \leq (t_2-t_1)\int_{\xi_1}^{\xi_2}\rho_{b_r}(\eta)u_{b_r}(\eta)^2 d \eta\,.
\]
Since the term in the integral is positive by our assumptions we conclude that the minimizers have to be increasing in $t$. \\
Now on the other hand fixing $t,\,x_1,\,x_2$ and denoting by $\xi_1$ a minimizer of $G_{b_r}(\xi,x_1,t)$ and by $\xi_2$ one of $G_{b_r}(\xi,x_2,t)$ we derive in the same way
\[
0\leq (x_1-x_2)\int_{\xi_1}^{\xi_2}\rho_{b_r}(\eta)u_{b_r}(\eta) d \eta\,.
\]
Using our assumption that $u_{b_r}$ is negative one infers that the minimizers are increasing in $x$.\\
For (4) note that 
\[
G_{b_r}(\xi,1,t)=\int_0^\xi u_{b_r}(\eta)^2\rho_{b_r}(\eta)(\eta-t)d\eta+F(1,1,t)\,,
\]
so the minimum is attained uniquely at $\xi=t$ because $\rho_{b_r}>0$.
\end{proof}
The following lemma gives uniqueness of the minimizers along lines connecting a point to the respective minimizer at the boundary corresponding to the potential. 
\begin{lem}\label{unique}
Let $(x,t)$ be a fixed point with $x\in ]0,1[$ and $t>0$.
\begin{enumerate}
\item Let $y=y_1$ be any minimizer of the functional $F(y,x, t)$. Then for any given point  $(x', t')\neq(x,t)$ on the line segment joining $(x,t)$ and $(y_1, 0)$, the minimizer of $F(y,x',t')$ is unique and it is $y=y_1$.
\item Let $\tau= \tau_{1}$ be any minimizer of the functional $G_{b_l}(\tau, x, t)$. Then for any point $(\bar x, \bar t )\neq(x,t)$ on the line segment joining $(x,t)$ and
$(0, \tau_{1})$ the minimizer of $G_{b_l}(\tau,\bar x, \bar t)$ is unique and it is $\tau_{1}$.
\item Let $\xi= \xi_{1}$ be any minimizer of the functional $G_{b_r}(\xi, x, t)$.  Then for any point $(\bar x, \bar t )\neq(x,t)$ on the line segment joining $(x,t)$ and
$(1, \xi_{1})$ the minimizer of $G_{b_r}(\xi,\bar x, \bar t)$ is unique and it is $\xi_{1}$.
\end{enumerate}
\end{lem}
\begin{proof}
The proof of (1) and (2) can be found in \cite{NOSS22}. Note that (1) was already included in \cite{WHD97} and \cite{Wang01}.
For (3): We want to show that for $\xi\neq\xi_1\colon$ 
\begin{equation*}
G_{b_r}(\xi, \bar{x}, \bar{t})-G_{b_r}(\xi_1, \bar{x}, \bar{t})>0\,.
\end{equation*}
By definition, we have
\begin{multline*}
G_{b_r}(\xi, \bar{x}, \bar{t})-G_{b_r}(\xi_1, \bar{x}, \bar{t})=\int_{\xi_1}^{\xi}
\left[\bar{x}-1-u_{b_r}(\eta)(\bar{t}-\eta)\right]\rho_{b_r}(\eta)u_{b_r}(\eta)d\eta = \\
=(\bar x-1)\int_{\xi_1}^{\xi} \left[1-u_{b_r}(\eta)\frac{\bar{t}-\eta}{\bar{x}-1}\right]\rho_{b_r}(\eta)u_{b_r}(\eta)d\eta=\\
=(\bar x-1)\int_{\xi_1}^{\xi} \left[1-u_{b_r}(\eta)\frac{\bar{t}-\xi_1}{\bar{x}-1}-u_{b_r}(\eta)\frac{\xi_1-\eta}{\bar{x}-1}\right]\rho_{b_r}(\eta)u_{b_r}(\eta)d\eta\,.
\end{multline*}
Using that $(\bar x,\bar t)$ lies on the line connecting $(x,t)$ and $(1,\xi_1)$ we conclude
\begin{multline*}
G_{b_r}(\xi, \bar{x}, \bar{t})-G_{b_r}(\xi_1, \bar{x}, \bar{t})=\\
=(\bar x-1)\int_{\xi_1}^{\xi} \left[1-u_{b_r}(\eta)\frac{t-\eta+\eta-\xi_1}{x-1}-u_{b_r}(\eta)\frac{\xi_1-\eta}{\bar{x}-1}\right]\rho_{b_r}(\eta)u_{b_r}(\eta)d\eta=\\
=(\bar x-1)\int_{\xi_1}^{\xi} \left[1-u_{b_r}(\eta)\frac{t-\eta}{x-1}\right]\rho_{b_r}(\eta)u_{b_r}(\eta)d\eta+\\
+(\bar x-1)\int_{\xi_1}^{\xi} (\xi_1-\eta)\left[\frac{1}{x-1}-\frac{1}{\bar x-1}\right]\rho_{b_r}(\eta)u_{b_r}^2(\eta)d\eta=\\
=\frac{\bar x-1}{ x-1}\left[G_{b_r}(\xi,x,t)-G_{b_r}(\xi_1,x,t)\right]+\int_{\xi_1}^{\xi} (\xi_1-\eta)\left[\frac{\bar x-1}{x-1}-1\right]\rho_{b_r}(\eta)u_{b_r}^2(\eta)d\eta 
\end{multline*}
Now the first term in the sum is non-negative by assumption and the second one is strictly positive. 
\end{proof}
\subsection{Characteristic triangles and their properties} Now we define the \textit{characteristic triangles} for any $(x, t)\in \Omega:=[0,1]\times [0, \infty).$ We begin with the definition of \textit{characteristic lines}.
\begin{defn}[Interior characteristic lines] For any point $(x_0, t_0)\in \Omega^o,$ we define three sets of characteristic lines:
\begin{align*}
    \begin{cases}
    X_l(x_0, t_0, t):=x=x_0+\frac{x_0-y_{*}(x_0, t_0)}{t_0}(t-t_0),\\
    X_r(x_0, t_0, t):=x=x_0+\frac{x_0-y^{*}(x_0, t_0)}{t_0}(t-t_0).
    \end{cases}
    \begin{cases}
        X^{b_l}_l:=x=x_0+\frac{x_0(t-t_0)}{t_0-\tau^*(x_0, t_0)},\\
        X^{b_l}_r:=x=x_0+\frac{x_0(t-t_0)}{t_0-\tau_{*}(x_0, t_0)}.
    \end{cases}
\end{align*}
\begin{align*}
    \begin{cases}
       X^{b_r}_l:=x=x_0+\frac{(x_0-1)(t-t_0)}{t_0-\xi_{*}(x_0, t_0)},\\
        X^{b_r}_r:=x=x_0+\frac{(x_0-1)(t-t_0)}{t_0-\xi^{*}(x_0, t_0)} .      
    \end{cases}
\end{align*}
\end{defn}

\begin{defn}[Boundary characteristic lines] Let $(x_0, t_0)\in \partial \Omega$ be any given point.
\begin{enumerate}
    \item If $(x_0, t_0)$ lies on the left boundary, i.e., $(x_0, t_0)=(0, t_0),$ we define the boundary characteristic lines as:
    \begin{align*}
        &(i)\,\,\, X_{b_l}(0, t_0, t):=x=\frac{y^*(0, t_0)}{t_0}(t_0-t),\\
        &(ii)\,\,\,X^{b_l}_{b_r}(1, 0, t_0, t):=x=\frac{t-t_0}{\xi^*(0, t_0)-t_0}.
    \end{align*}
    \item If $(x_0, t_0)$ lies on the right boundary, i.e., $(x_0, t_0)=(1, t_0),$ we define boundary characteristic lines as:
    \begin{align*}
        &(i)\,\,\, X_{b_r}(1, t_0, t):=x=1+\frac{y_*(1, t_0)-1}{t_0}(t_0-t),\\
        &(ii)\,\,\, X^{b_r}_{b_l}(1, 0, t_0, t):=1-x=\frac{t-t_0}{\tau^*(1, t_0)-t_0}.
    \end{align*}
\end{enumerate}
\end{defn}
Before defining the characteristic triangles, we need to define the following sets, which we call the \textit{Interior characteristic sets.} Using the interior characteristic lines, we define three types of sets as:
\begin{align*} (I)\,
\begin{cases}
    X^{+}_l:=\left\{x\in (0, 1), 0<t\leq t_0\,\,\Big|\,\, X_l(x_0, t_0, t)\leq x\right\}\\
     X^{-}_r:=\left\{x\in (0, 1), 0<t\leq t_0\,\,\Big|\,\, X_r(x_0, t_0, t)\geq x\right\},
 \end{cases}    
\end{align*}
\begin{align*} (LB)\,
\begin{cases}
    X^{b_l, +}_l:=\left\{x\in (0, 1), 0<t\leq t_0\,\,\Big|\,\, X^{b_l}_l(x_0, t_0, t)\leq x\right\}\\
    X^{b_l, -}_r:=\left\{x\in (0, 1), 0<t\leq t_0\,\,\Big|\,\, X^{b_l}_r(x_0, t_0, t)\geq x\right\},
 \end{cases}   
\end{align*}
\begin{align*} (RB)\,
\begin{cases}
 X^{b_r, +}_l:=\left\{x\in (0, 1), 0<t\leq t_0\,\,\Big|\,\, X^{b_r}_l(x_0, t_0, t)\leq x\right\}\\
X^{b_r, -}_r:=\left\{x\in (0, 1), 0<t\leq t_0\,\,\Big|\,\, X^{b_r}_r(x_0, t_0, t)\geq x\right\}.
\end{cases}    
\end{align*}
\begin{defn}[Interior characteristic triangles] Let $0<x_0<1$ and $t_0>0$ and $F(x_0, t_0), G_{b_l}(x_0, t_0),$ $G_{b_r}(x_0, t_0)$ are given by \eqref{InitPot}, \eqref{leftBdpot}, \eqref{rightBdpot}.
\begin{enumerate}
    \item[(i)] If $F(x_0, t_0)<\min \left\{G_{b_l}(x_0, t_0), G_{b_r}(x_0, t_0)\right\},$ then we define the interior characteristic triangle at $(x_0, t_0)$ as the area enclosed by the region $X^{+}_l\cap X^{-}_r.$
    \item[(ii)] If $G_{b_l}(x_0, t_0)< \min\left\{F(x_0, t_0), G_{b_r}(x_0, t_0)\right\},$ then we define the interior characteristic triangle at $(x_0, t_0)$ as the area enclosed by the region $X^{b_l, +}_l\cap X^{b_l, -}_r.$
    \item[(iii)] If $G_{b_r}(x_0, t_0)< \min\left\{F(x_0, t_0), G_{b_l}(x_0, t_0)\right\},$ then we define the interior characteristic triangle at $(x_0, t_0)$ as the area enclosed by the region $X^{b_r, +}_l\cap X^{b_r, -}_r$
    \item[(iv)] If $G_{b_l}(x_0, t_0)=G_{b_r}(x_0, t_0),$ then we define the characteristic triangle at $(x_0, t_0)$ as the area enclosed by the region $X^{b_l, +}_l\cap X^{b_r, -}_r.$
    \item[(v)] If $F(x_0, t_0)=G_{b_l}(x_0, t_0),$ then we define the interior characteristic triangle at $(x_0, t_0)$ as the area enclosed by the region $X^{b_l, +}_l\cap X^{-}_r.$
    \item[(vi)] If $F(x_0, t_0)=G_{b_r}(x_0, t_0),$ then we define the interior characteristic triangle at $(x_0, t_0)$ as the area enclosed by the region $X^{+}_l\cap X^{b_r, -}_r.$ 
\end{enumerate}
Moreover, interior characteristic triangles associated to $(x_0, t_0)\in \Omega^{o}$ is denoted as $\Delta(x_0, t_0).$
\end{defn}
\begin{figure}[htb]
\centering
\begin{tikzpicture}[scale=0.8]
\draw[-] (0,0) -- (9,0) node[anchor=north] {$x=1$};
\draw[-] (9,0) -- (0,0) node[anchor=north] {$x=0$};
\draw[->] (0,0) -- (0,6) node[anchor=east] {$t$};
\draw[->] (9,0) -- (9,6) node[anchor=east] {$t$};
\draw (0,4) node[anchor=east]{$T$};
\draw[dashed] (0,4) -- (9,4);
\draw (1,4) node[anchor=south] {$x_1$};
\draw	(1,4) circle[radius=1.5pt];
\fill (1,4) circle[radius=1.5pt];
\draw[pattern=north west lines, pattern color=teal] (0,2) to (1,4) to (0,3) to (0,2);
\draw (0,2) node[anchor=east] {$\tau_*(x_1,T)$};
\draw	(0,2) circle[radius=1.5pt];
\draw (0,3) node[anchor=east] {$\tau^*(x_1,T)$};
\draw	(0,3) circle[radius=1.5pt];
\draw (2.5,4) node[anchor=south] {$x_2$};
\draw	(2.5,4) circle[radius=1.5pt];
\fill (2.5,4) circle[radius=1.5pt];
\draw[pattern=horizontal lines, pattern color=teal] (0,1) to (2.5,4) to (2,0) to (0,0) to (0,1);
\draw (0,1) node[anchor=east] {$\tau^*(x_2,T)$};
\draw	(0,1) circle[radius=1.5pt];
\node [anchor=center,rotate=90] at (2,-1) {$y^*(x_2,T)$};
\draw	(2,0) circle[radius=1.5pt];
\draw (4,4) node[anchor=south] {$x_3$};
\draw	(4,4) circle[radius=1.5pt];
\fill (4,4) circle[radius=1.5pt];
\draw[pattern=north east lines, pattern color=teal] (3,0) to (4,4) to (5,0) to (3,0);
\node [anchor=center,rotate=90] at (3,-1) {$y_*(x_3,T)$};
\draw	(3,0) circle[radius=1.5pt];
\node [anchor=center,rotate=90] at (5,-1) {$y^*(x_3,T)$};
\draw	(5,0) circle[radius=1.5pt];
\draw (6,4) node[anchor=south] {$x_4$};
\draw	(6,4) circle[radius=1.5pt];
\fill (6,4) circle[radius=1.5pt];
\draw[pattern=horizontal lines, pattern color=teal] (9,1) to (6,4) to (7,0) to (9,0) to (9,1);
\draw (11.2,1) node[anchor=east] {$\xi^*(x_4,T)$};
\draw	(9,1) circle[radius=1.5pt];
\node [anchor=center,rotate=90] at (7,-1) {$y_*(x_4,T)$};
\draw	(7,0) circle[radius=1.5pt];
\draw (7.5,4) node[anchor=south] {$x_5$};
\draw	(7.5,4) circle[radius=1.5pt];
\fill (7.5,4) circle[radius=1.5pt];
\draw[pattern=north east lines, pattern color=teal] (9,2) to (7.5,4) to (9,3) to (9,2);
\draw (11.2,2) node[anchor=east] {$\xi_*(x_5,T)$};
\draw	(9,2) circle[radius=1.5pt];
\draw (11.2,3) node[anchor=east] {$\xi^*(x_5,T)$};
\draw	(9,3) circle[radius=1.5pt];
\end{tikzpicture}
\vspace*{-.3cm}\caption{Figures of interior characteristic triangles (i)-(v)}\label{bildtriangle_1}
\end{figure}
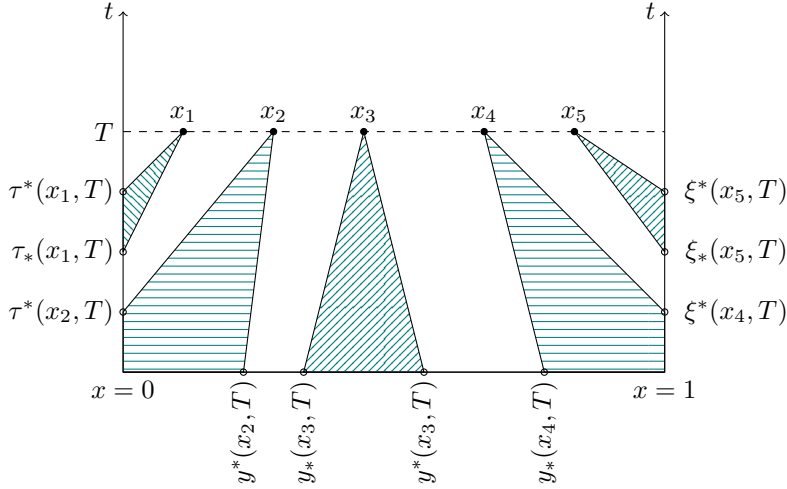

\begin{defn}[Boundary characteristic triangles] Let $(x_0, t_0) \in \partial \Omega$ be any point. 
\begin{enumerate}
    \item [(i)] If $F(0, t_0)< G_{b_l}(0,t_0),$ then we define the boundary characteristic triangle as the area enclosed by the set $X^{bd}_l:=\left\{x\in [0, 1], 0\leq t \leq t_0 \,\, \Big|\,\, x \leq X_{b_l}(0, t_0, t)\right\}.$
    \item [(ii)] If $F(1, t_0)< G_{b_r}(1, t_0),$ then we define the boundary characteristic triangle as the area enclosed by the set $X^{bd}_r:=\left\{x\in [0, 1], 0\leq t \leq t_0 \,\,\Big|\,\, x \geq X_{b_l}(0, t_0, t)\right\}.$
    \item[(iii)] If $G_{b_r}(0,t_0)<G_{b_l}(0, t_0),$ then we define the boundary characteristic triangle as the area enclosed by the set $X^{l}_r:=\left\{x\in [0, 1], 0\leq t \leq t_0 \,\,\Big|\,\, x \leq X^{b_l}_{b_r}(1, t_0, t)\right\}.$
    \item [(iv)] If $G_{b_l}(1, t_0)< G_{b_l}(1, t_0,)$ then we define the boundary characteristic triangle as the area enclosed by the set $X^{r}_l:=\left\{x\in [0, 1], 0\leq t \leq t_0 \,\,\Big|\,\, x \leq X^{b_r}_{b_l}(1, t_0, t)\right\}.$
\end{enumerate}
Also, the boundary characteristic triangles associated to $(0, t_0)$ or $(1, t_0)$ and also denoted as $\Delta(x_0, t_0).$
\end{defn}
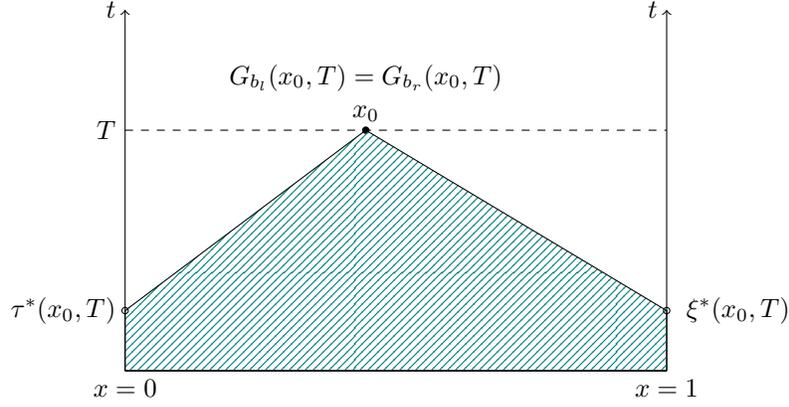
\begin{figure}[htb]
\centering
\begin{tikzpicture}[scale=0.8]
\draw[-] (0,0) -- (9,0) node[anchor=north] {$x=1$};
\draw[-] (9,0) -- (0,0) node[anchor=north] {$x=0$};
\draw[->] (0,0) -- (0,6) node[anchor=east] {$t$};
\draw[->] (9,0) -- (9,6) node[anchor=east] {$t$};
\draw (0,4) node[anchor=east]{$T$};
\draw[dashed] (0,4) -- (9,4);
\draw (2,0) node[anchor=north] {$ $}
(4,4.5) node[anchor=south] {$G_{b_l}(x_0, T)=G_{b_r}(x_0, T)$};
			
%
%
\draw (4,4) node[anchor=south] {$x_0$};
\draw	(4,4) circle[radius=1.5pt];
\fill (4,4) circle[radius=1.5pt];
\draw[pattern=north east lines, pattern color=teal] (0,1) to (0,0) to (9, 0) to (9,1) to (4, 4) to (0, 1);
\draw (0,1) node[anchor=east] {$\tau^*(x_0,T)$};
\draw	(0,1) circle[radius=1.5pt];
\draw (11.2,1) node[anchor=east] {$\xi^*(x_0,T)$};
\draw	(9,1) circle[radius=1.5pt];
%
\end{tikzpicture}
\vspace*{-.3cm}\caption{Figure of interior characteristic triangle (vi)}\label{bildtriangle_2}
\end{figure}

\begin{figure}[htb]
\centering
\begin{tikzpicture}[scale=0.8]
\draw[-] (0,0) -- (9,0) node[anchor=north] {$x=1$};
\draw[-] (9,0) -- (0,0) node[anchor=north] {$x=0$};
\draw[->] (0,0) -- (0,6) node[anchor=east] {$t$};
\draw[->] (9,0) -- (9,6) node[anchor=east] {$t$};
\draw (0,4) node[anchor=east]{$T$};
\draw[dashed] (0,4) -- (9,4);
\draw (2,0) node[anchor=north] {$ $}
		(2,4.5) node[anchor=south] {$F(0, T)<G_{b_l}(0, T)$}
		(7,4.5) node[anchor=south] {$F(1, T)<G_{b_r}(1, T)$};
\draw (0,4) node[anchor=south] {};
\draw	(0,4) circle[radius=1.5pt];
\fill (0,4) circle[radius=1.5pt];
\draw[pattern=north east lines, pattern color=brown] (0,4) to (0,0) to (3.5,0) to (0,4);
\draw (4,-0.5) node[anchor=east] {$y^*(0,T)$};
\draw	(3.5,0) circle[radius=1.5pt];
\draw (9,4) node[anchor=south] {};
\draw	(9,4) circle[radius=1.5pt];
\fill (9,4) circle[radius=1.5pt];
\draw[pattern=north west lines, pattern color=brown] (9,4) to (9, 0) to (5,0) to (9,4);
\draw (6,-0.5) node[anchor=east] {$y_*(1,T)$};
\draw	(5,0) circle[radius=1.5pt];
%
%
%
\end{tikzpicture}
\vspace*{-.3cm}\caption{Figures of boundary characteristic triangles (i)-(ii)}\label{bildtriangle_3}
\end{figure}
\begin{figure}[htb]
\centering
\begin{tikzpicture}[scale=0.8]
\draw[-] (0,0) -- (9,0) node[anchor=north] {$x=1$};
\draw[-] (9,0) -- (0,0) node[anchor=north] {$x=0$};
\draw[->] (0,0) -- (0,6) node[anchor=east] {$t$};
\draw[->] (9,0) -- (9,6) node[anchor=east] {$t$};
\draw (0,4) node[anchor=east]{$T$};
\draw[dashed] (0,4) -- (9,4);
\draw (2,0) node[anchor=north] {$ $}
			(2.3,4.5) node[anchor=south] {$G_{b_r}(0, T)<G_{b_l}(0, T)$};
\draw (0,4) node[anchor=south] {};
\draw	(0,4) circle[radius=1.5pt];
\fill (0,4) circle[radius=1.5pt];
\draw[pattern=north east lines, pattern color=brown] (0,4) to (0,0) to (9,0) to (9,1) to (0,4);
\draw (11, 1) node[anchor=east] {$\xi^*(0,T)$};
\draw	(9,1) circle[radius=1.5pt];
\end{tikzpicture}
\vspace*{-.3cm}\caption{Figure of boundary characteristic triangle (iii)}\label{bildtriangle_4}
\end{figure}
\begin{figure}[htb]
\centering
\begin{tikzpicture}[scale=0.8]
\draw[-] (0,0) -- (9,0) node[anchor=north] {$x=1$};
\draw[-] (9,0) -- (0,0) node[anchor=north] {$x=0$};
\draw[->] (0,0) -- (0,6) node[anchor=east] {$t$};
\draw[->] (9,0) -- (9,6) node[anchor=east] {$t$};
\draw (0,4) node[anchor=east]{$T$};
\draw[dashed] (0,4) -- (9,4);
\draw (2,0) node[anchor=north] {$ $}
			(6.7,4.5) node[anchor=south] {$G_{b_l}(1, T)<G_{b_r}(1, T)$};
%
\draw (9,4) node[anchor=south] {};
\draw	(9,4) circle[radius=1.5pt];
\fill (9,4) circle[radius=1.5pt];
\draw[pattern=north west lines, pattern color=brown] (0,1) to (0, 0) to (9,0) to (9,4) to (0,1);
\draw (-0.2,1) node[anchor=east] {$\tau^*(1,T)$};
\draw	(0,1) circle[radius=1.5pt];
%
%
%
\end{tikzpicture}
\vspace*{-.3cm}\caption{Figure of boundary characteristic triangle (iv)}\label{bildtriangle_4}
\end{figure}

In the following lemma, we collect some properties of the characteristic triangles. The proof is similar to \cite[Lemma 2.9-2.13]{NOSS22}.
\begin{lem}
Grant the assumptions on the initial and boundary data. Then, we have
\begin{enumerate}
    \item Let $t > 0$ be fixed, and  $x_1 \neq  x_2$ be arbitrary. Then the characteristic triangles
associated with $(x_1,t)$ and $(x_2,t)$ do not intersect in $[0,1] \times [0, \infty).$ Consequently, if two characteristic triangles intersect in $[0, 1]\times [0, \infty)$, then one is contained in the other.
\item For any time $t_0>0$, we have
\begin{align*}
    \bigcup_{x \in [0,1]}\Delta(x,t_0)=\left\{(x,t)\big|x \in [0,1],0 \leq t \leq t_0\right\}.
\end{align*}
\end{enumerate}
\end{lem}

\begin{lem}\label{lem:curves}
Let $t_1>0$ be given. Each point $(x_1, t_1)$ uniquely determines a curve $x=X(t)$, for $t\geq t_1$, with $x_1= X(t_1)$ such that the characteristic triangles associated to points on the curve forms an increasing family of sets.
Furthermore, the curve
\begin{align*}
[t_1, \infty)\to [0, 1], \,\,\,\, t\mapsto X(t)
\end{align*}
is at least of class $C^0$ and for every $t\geq t_1$, and $(x,t)$ on the curve we have the following:
\begin{enumerate}[label=(\roman*)]
\item \label{it1} If $F(x,t)<\min\{G_{b_l}(x,t), G_{b_r}(x,t)\}$ then
\begin{equation*}\label{e2.7}
\lim_{t'', t'\searrow t}\frac{X(t'')-X(t')}{t''-t'}=
\left\{
\begin{aligned}
&\frac{x-y_*}{t} &&\text{if}&&y_*(x,t)=y^*(x,t),\,\\
&\displaystyle{\frac{\int_{y_{*}}^{y^{*}}\rho_0 u_0}{\int_{y_*}^{y^*}\rho_0} }&&\text{if}&&y_*(x,t)<y^*(x,t)\,.
\end{aligned}
\right.
\end{equation*}
\item \label{it2} If $G_{b_l}(x,t) < \min\{F(x,t), G_{b_r}(x,t)\}$ then
\begin{equation*}\label{e2.8}
\lim_{t'', t'\searrow t}\frac{X(t'')-X(t')}{t''-t'}=
\left\{
\begin{aligned}
&\frac{x}{t-\tau_*} &&\text{if}&&\tau_{*}(x,t)=\tau^{*}(x,t),\,\\
&\displaystyle{\frac{\int_{\tau_*}^{\tau^*}\rho_{b_l} u^2_{b_l}}{\int_{\tau_*}^{\tau^*}u_{b_l}\rho_{b_l}} }&&\text{if}&&\tau_{*}(x,t)<\tau^{*}(x,t)\,.
\end{aligned}
\right.
\end{equation*}
\item \label{it3} If $\{F(x,t)=G_{b_l}(x,t)\}$ and $y^*(x,t)\neq0$ or $\tau^*(x,t)\neq 0$ then
\begin{equation*}\label{e2.9}
\lim_{t'', t'\searrow t}\frac{X(t'')-X(t')}{t''-t'}=
\frac{\int_{0}^{y^*}\rho_0 u_{0}+\int_{0}^{\tau ^*}\rho_{b_l} u^2_{b_l}}{\int_{0}^{y^*}\rho_0+\int_{0}^{\tau^*}u_{b_l}\rho_{b_l}}\,.
\end{equation*}
\item \label{it4} If $F(x,t)=G_{b_l}(x,t)$ and $y^*(x,t)=\tau^*(x,t)=0$, then
\begin{equation*}\label{it4}
\lim_{t'', t'\searrow t}\frac{X(t'')-X(t')}{t''-t'}=\frac{x}{t}.
\end{equation*}
\item \label{it5} If $x=0$, $t>0$ and $\min\left\{F(0,t), G_{b_r}(0, t)\right\}<G_{b_l}(0,t)$, then
\begin{equation*}
\lim_{t'', t'\searrow t}\frac{X(t'')-X(t')}{t''-t'}=0.
\end{equation*}

\item \label{it6} If $G_{b_r}(x,t) < \min\{F(x,t), G_{b_l}(x,t)\}$ then
\begin{equation*}\label{newcase1}
\lim_{t'', t'\searrow t}\frac{X(t'')-X(t')}{t''-t'}=
\left\{
\begin{aligned}
&\frac{x-1}{t-\xi_*} &&\text{if}&&\xi_*(x,t)=\xi^*(x,t),\,\\
&\displaystyle{\frac{\int_{\xi_{*}}^{\xi^{*}}\rho_{b_r} u^2_{b_r}}{\int_{\xi_*}^{\xi^*}\rho_{b_r}u_{b_r}} }&&\text{if}&&\xi_*(x,t)<\xi^*(x,t)\,.
\end{aligned}
\right.
\end{equation*}
\item \label{it7} If $\{F(x,t)=G_{b_r}(x,t)\}$ and $y_*(x,t)\neq0$ or $\tau^*(x,t)\neq 0$ then
\begin{equation*}\label{newcase2}
\lim_{t'', t'\searrow t}\frac{X(t'')-X(t')}{t''-t'}=
\frac{\int_{0}^{y_*}\rho_0 u_{0}+\int_{0}^{\xi ^*}\rho_{b_r} u^2_{b_r}-\int_0^1 \rho_0 u_0}{\int_{0}^{y_*}\rho_0+\int_{0}^{\xi^*}u_{b_r}\rho_{b_r}-\int_0^1 \rho_0}\,.
\end{equation*}
\item \label{it8} If $\{G_{b_l}(x,t)=G_{b_r}(x,t)\},$ then
\begin{equation*}\label{newcase3}
\lim_{t'', t'\searrow t}\frac{X(t'')-X(t')}{t''-t'}=
\frac{\int_{0}^{\tau^*}\rho_{b_l} u^2_{b_l}-\int_{0}^{\xi ^*}\rho_{b_r} u^2_{b_r}+\int_0^1 \rho_0 u_0}{\int_{0}^{\tau^*}\rho_{b_l} u_{b_l}-\int_{0}^{\xi^*}u_{b_r}\rho_{b_r}+\int_0^1 \rho_0}\,.
\end{equation*}
\item \label{it9} If $F(x,t)=G_{b_r}(x,t)$ and $y^*(x,t)=1, \xi^*(x,t)=0$, then
\begin{equation*}
\lim_{t'', t'\searrow t}\frac{X(t'')-X(t')}{t''-t'}=\frac{x-1}{t}.
\end{equation*}
\item \label{it10} If $x=1$, $t>0$ and $\min\left\{F(1,t), G_{b_l}(1, t)\right\}<G_{b_r}(1,t)$, then
\begin{equation*}
\lim_{t'', t'\searrow t}\frac{X(t'')-X(t')}{t''-t'}=0.
\end{equation*}
\end{enumerate}
\end{lem}

\begin{proof}
The statement of \ref{it1} corresponds to Lemma~2.4 in \cite{WHD97} and the proof can be found there.
The statements \ref{it1}-\ref{it5} are proved in \cite{NOSS22}. Following a similar strategy we prove \ref{it6}-\ref{it10}. 
For the proof of \ref{it6}, let $t''>t'>t$ and $X(t'')=x'', X(t')=x', X(t)=x.$
First consider the case when $\xi ^{*}(x,t)=\xi_{*}(x,t)$. It is straightforward to see the following inequality on inclinations:
\begin{equation}\label{e2.10}
\frac{x''-1}{t''-\xi^{*}(x'', t'')} \geq \frac{x''-x'}{t''-t'} \geq \frac{x''-1}{t''-\xi_{*}(x'', t'')}.
\end{equation}
Passing to the limit as $t'', t' \searrow t$ leads, using the semicontinuity and monotonicity, to the first identity of equation \ref{it6}.\\
Now we consider the case $\xi_{*}(x,t)<\xi ^{*}(x,t)$. From the definitions of the boundary functional $G_{b_r}(\xi, x, t)$, $\xi_*$ and $\xi^*$, we have
\begin{equation}\label{e2.11}
\begin{split}
G_{b_r}\big(\xi^{*}(x'', t''), x'', t''\big) - G_{b_r}\big(\xi_{*}(x', t'), x'', t''\big)\leq 0\leq G_{b_r}\big(\xi^{*}(x'', t''), x', t'\big)- G_{b_r}\big(\xi_{*}(x', t'),  x', t'\big).
\end{split}
\end{equation}
After simplification inequality \eqref{e2.11} yields
\begin{equation*}
\int\displaylimits_{\xi_{*}(x', t')} ^ {\xi^{*}(x'', t'')} \!\!\![x''-1-u_{b_r} (\eta)(t''-\eta)] \rho_{b_r} (\eta)u_{b_r}(\eta) d \eta \leq \int\displaylimits_{\xi_{*}(x', t')} ^ {\xi^{*}(x'', t'')} \!\!\![x'-1-u_{b_r} (\eta)(t'-\eta)] \rho_{b_r} (\eta)u_{b_r}(\eta) d \eta\,.
\end{equation*}
Thus we can conclude
\begin{equation}\label{e2.13}
\frac{x''- x'}{t''- t'} \leq
\frac{\int_{\xi_{*}(x', t')} ^ {\xi^{*}(x'', t'')}u^2_{b_r} (\eta)\rho_{b_r} (\eta)d \eta}{\int_{\xi_{*}(x', t')} ^ {\xi^{*}(x'', t'')}u_{b_r}(\eta)\rho_{b_r} (\eta)d \eta}\,.
\end{equation}
On the other hand, considering the inequality
\begin{equation*}
\begin{split}
G_{b_r}\big(\xi_{*}(x'', t''), x'', t''\big) -G_{b_r}\big(\xi^{*}(x', t'), x'', t''\big) \leq 0 \leq G_{b_r}\big(\xi_{*}(x'', t''), x', t'\big)- G_{b_r}\big(\xi^{*}(x', t'),x',t'\big),
\end{split}
\end{equation*}
we get
\begin{equation}\label{e2.14}
\frac{x''- x'}{t''- t'} \geq
\frac{\int_{\tau_{*}(x', t')} ^ {\tau^{*}(x'', t'')}u^2_{b_r} (\eta)\rho_{b_r} (\eta)d \eta}{\int_{\tau_{*}(x', t')} ^ {\tau^{*}(x'', t'')}u_{b_r}(\eta)\rho_{b_r} (\eta)d \eta}\,.
\end{equation}
Now passing to the limit as $t'', t'\searrow t$ in equations \eqref{e2.13}  and  \eqref{e2.14}, we proved the second identity of \ref{newcase1}.

Now to verify the statement of \ref{it7}, we assume $F(x,t)=G_{b_r}(x,t)$. Then by definition of the curve and the characteristic triangles, we find $(t, t+\varepsilon)$ such that the minima $F$ and $G_{b_r}$ are equal along the curve for some $t^{\prime \prime}, t^{\prime}\in (t. t+\varepsilon)$ with $t^{\prime \prime}>t^{\prime}$. Using the minimizing properties we have the following inequality.
\begin{align*}
&F\big(y_{*}({x'', t''}), x'', t''\big)-G_{b_r}\big(\xi^{*}(x',t') , x'', t''\big)
=G_{b_r}\big(\xi^{*}({x'', t''}), x'', t''\big)-G_{b_r}\big(\xi^{*}(x',t') , x'', t''\big)\leq 0 \\
&\leq F\big(y_{*}(x'', t''), x', t'\big)-F\big(y_ {*}(x',t'),  x', t'\big)
=F\big(y_{*}(x'', t''), x', t'\big)-G_{b_r}\big(\xi^ {*}(x',t'),  x', t'\big).
\end{align*}
This inequality implies
\begin{equation*}
\begin{split}
&\int_{0}^{y_{*}({x'', t''})}(t^{\prime \prime}-t^{\prime})u_0(\eta)\rho_0(\eta)+(x^{\prime}-x^{\prime \prime})\rho_0(\eta)d\eta + \int_0^1 (t^{\prime}-t^{\prime\prime})\rho_0(\eta) u_0(\eta)+(x^{\prime\prime}-x^{\prime})\rho_0(\eta)d\eta\\
&\leq \int_{0}^{\xi^{*}(x',t')}(x''-x')\rho_{b_r}(\eta)u_{b_r}(\eta)+ u_{b_r}^{2}(\eta)\rho_{b_r}(\eta)(t'-t'')d\eta
\end{split}
\end{equation*}
and a simplification leads to
\begin{equation}\label{e2.17}
\frac{\int_{0}^{y_{*}({x'', t''})}\rho_0 u_{0}+\int_{0} ^{\xi^{*}(x',t')}\rho_{b_r} u^2_{b_r}-\int_0^1 \rho_0 u_0}{\int_{0}^{y_{*}(x'', t'')}\rho_0+\int_{0}^{\xi^{*}(x',t')}u_{b_r}\rho_{b_r}-\int_0^1 \rho_0}\leq \frac{x''-x'}{t''-t'}\,.
\end{equation}
Now in the same way starting from the inequality
\begin{equation*}
\begin{split}
G_{b_r}\big(\xi^{*}(x'', t''), x'', t'' \big)-&F\big(y_{*}({x', t'}), x'', t''\big)\leq G_{b_r}\big(\xi^ {*}(x'', t''), x', t'\big)- F\big(y_{*}(x', t'), x', t'\big)
\end{split}
\end{equation*}
and simplifying as before we get
\begin{equation}\label{e2.18}
\frac{\int_{0}^{y_{*}({x'', t''})}\rho_0 u_{0}+\int_{0} ^{\xi^{*}(x',t')}\rho_{b_r} u^2_{b_r}-\int_0^1 \rho_0 u_0}{\int_{0}^{y_{*}({x'', t''})}\rho_0+\int_{0}^{\xi^{*}(x',t')}u_{b_r}\rho_{b_r}-\int_0^1 \rho_0}\geq\frac{x''-x'}{t''-t'}\,.
\end{equation}
Passing to the limit as $t'', t'\searrow t$  in \eqref{e2.17} and \eqref{e2.18} completes the proof of \ref{it7}.\\
Note that the denominator of the above inequality is never zero, since
$\int_0^{y^*(x,t)}\rho_0(\eta)d\eta-\int_0^1 \rho_0(\eta)d\eta \leq 0$ and $\int_0^{\xi^*(x,t)}u_{b_r}(\eta)\rho_{b_r}(\eta)d\eta <0.$

Following the similar arguments, we prove \ref{it8}. Consider the inequality
\begin{equation*}
\begin{aligned}
&G_{b_r}\left(\xi^*(x^{\prime \prime}, t^{\prime \prime}, x^{\prime}, t^{\prime})\right)-G_{b_r}\left(\xi^*(x^{\prime}, t^{\prime}), x^{\prime \prime}, t^{\prime, \prime}\right)=G_{b_l}(\tau^*(x^{\prime\prime}, t^{\prime\prime}), x^{\prime\prime}, t^{\prime\prime})-G_{b_r}(\xi^{*}(x^{\prime}, t^{\prime}), x^{\prime \prime}, t^{\prime \prime})\leq 0\\
&\leq G_{b_l}(\tau^*(x^{\prime\prime}, t^{\prime\prime}), x^{\prime}, t^{\prime})-G_{b_l}(\tau^{*}(x^{\prime}, t^{\prime}), x^{\prime}, t^{\prime})= G_{b_l}(\tau^*(x^{\prime\prime}, t^{\prime\prime}), x^{\prime}, t^{\prime})-G_{b_r}(\xi^{*}(x^{\prime}, t^{\prime}), x^{\prime}, t^{\prime}).
\end{aligned}
\end{equation*}
Simplifying the above inequality we find
\begin{equation*}
\begin{aligned}
&(x^{\prime\prime}-x^{\prime})\Big[\int_0^{\tau^*(x^{\prime \prime}, t^{\prime\prime})}u_{b_l}(\eta)\rho_{b_l}(\eta)d\eta-\int_0^{\xi^*(x^{\prime},t^{\prime})}u_{b_r}(\eta)\rho_{b_r}(\eta)d\eta+\int_0^1 \rho_0(\eta)d\eta\Big]\\
&\leq (t^{\prime\prime}-t^{\prime})\Big[\int_0^{\tau^*(x^{\prime \prime}, t^{\prime\prime})}u^2_{b_l}(\eta)\rho_{b_l}(\eta)d\eta-\int_0^{\xi^*(x^{\prime},t^{\prime})}u^2_{b_r}(\eta)\rho_{b_r}(\eta)d\eta+\int_0^1 \rho_0(\eta)u_0(\eta)d\eta\Big].
\end{aligned}
\end{equation*}
Therefore, we have
\begin{equation}\label{new equ2.17}
\frac{x^{\prime\prime}-x^{\prime}}{t^{\prime \prime}-t^{\prime}} \leq \frac{\int_0^{\tau^*(x^{\prime \prime}, t^{\prime\prime})}u^2_{b_l}(\eta)\rho_{b_l}(\eta)d\eta-\int_0^{\xi^*(x^{\prime},t^{\prime})}u^2_{b_r}(\eta)\rho_{b_r}(\eta)d\eta+\int_0^1 \rho_0(\eta)u_0(\eta)d\eta}{\int_0^{\tau^*(x^{\prime \prime}, t^{\prime\prime})}u_{b_l}(\eta)\rho_{b_l}(\eta)d\eta-\int_0^{\xi^*(x^{\prime},t^{\prime})}u_{b_r}(\eta)\rho_{b_r}(\eta)d\eta+\int_0^1 \rho_0(\eta)d\eta}.
\end{equation}
Similarly considering
\begin{equation*}
G_{b_l}(\tau^*(x^{\prime},t^{\prime}),x^{\prime}, t^{\prime})-G_{b_r}(\xi^*(x^{\prime\prime},t^{\prime\prime}), x^{\prime}, t^{\prime}) \leq G_{b_l}(\tau^*(x^{\prime},t^{\prime}), x^{\prime\prime},t^{\prime\prime})-G_{b_r}(\xi^*(x^{\prime\prime}, t^{\prime\prime}),x^{\prime\prime}, t^{\prime\prime}),
\end{equation*}
we get
\begin{equation}\label{new equ2.18}
\frac{x^{\prime\prime}-x^{\prime}}{t^{\prime \prime}-t^{\prime}} \geq \frac{\int_0^{\tau^*(x^{\prime \prime}, t^{\prime\prime})}u^2_{b_r}(\eta)\rho_{b_r}(\eta)d\eta-\int_0^{\xi^*(x^{\prime},t^{\prime})}u^2_{b_r}(\eta)\rho_{b_r}(\eta)d\eta+\int_0^1 \rho_0(\eta)u_0(\eta)d\eta}{\int_0^{\tau^*(x^{\prime \prime}, t^{\prime\prime})}u_{b_l}(\eta)\rho_{b_l}(\eta)d\eta-\int_0^{\xi^*(x^{\prime},t^{\prime})}u_{b_r}(\eta)\rho_{b_r}(\eta)d\eta+\int_0^1 \rho_0(\eta)d\eta}.
\end{equation}
In the inequalities \eqref{new equ2.17} and \eqref{new equ2.18}, we observe that the denominator can not be zero, since
\[\int_0^{\tau^*(x,t)}u_{b_l}(\eta)\rho_{b_l}(\eta)d\eta-\int_0^{\xi^*(x,t)}u_{b_r}(\eta)\rho_{b_r}(\eta)d\eta \geq 0\,\,\, \text{and }\,\,\, \int_0^1 \rho_0(\eta)d\eta >0.\] The proof of \ref{it9} and \ref{it10} is similar to the proof of \ref{it4}, \ref{it5} and follows from the definition of characteristic triangles.  This completes the proof.
\end{proof}
\begin{rem}
In fact, the curve $X(t)$ is Lipschitz continuous. Indeed, \eqref{e2.10}, \eqref{e2.13}-\eqref{new equ2.18} shows that for every $t^{\prime\prime}>t^{\prime},$ we have $|X(t^{\prime\prime})-X(t^{\prime})|\leq L|t^{\prime\prime}-t^{\prime}|$.
\end{rem}
\begin{rem}\label{rem2.9}
We comment on the regions where $F(x, t)=G_{b_r}(x, t)$ and $G_{b_l}(x, t)=G_{b_r}(x, t).$ By \cite[Remark 2.2]{NOSS22}, we have that if $\min_{\tau \geq 0}G_{b_l}(\tau, x, t)$ is constant on an interval $[x_1, x_2]\times \{t=T\},$ then $G_{b_l}(x, T)=\min_{\tau \geq 0} G_{b_l}(\tau, x, T)=0$ for all $x \in [x_1, x_2].$ Define the sets: 
\begin{align*}
    I_1(T):=\left\{x \,\,\big|\,\, G_{b_l}(x, T)=G_{b_r}(x, T)\right\},\,\,\,I_2(T):=\left\{x \,\,\big|\,\, F(x, T)=G_{b_r}(x, T)\right\}.
\end{align*}
Now we recall that for any $t>0,$ $G_{b_l}(x, t)$ increases with $x$ while $G_{b_r}(x, t)$ decreases with $x.$ Thus, $G_{b_l}(x, t)=G_{b_r}(x, t)$ implies both are constant and for all $x \in I_1(T),$ $G_{b_r}(x, T)=0.$ 

\noindent \textit{Claim 1:} $I_1(T):=\{x\}.$ Suppose $x_1, x_2 \in I_1(T)$ with $x_1<x_2.$ From \cite[Remark 2.2]{NOSS22}, it follows that $\tau^*(x_1, T)=\tau_*(x_1, T)=\tau^*(x_2, T)=\tau_*(x_2, T)=0$ and $\xi^*(x_1, T), \xi_*(x_1, T), \xi^*(x_2, T), \xi_*(x_2, T)\in [0, T]$ with $\xi^*(x_1, T)<\xi^*(x_2, T).$ Then the line segment joining $(x_2, T)$ to $(0, 0)$ intersects the line segment joining $(x_1, T)$ to $(1, \xi^*(x_1, T))$ at some point $(x_p, t)$ which is a contradiction to the Lemma \ref{unique}. Hence $I_1(T)=\{x\}.$ 

\noindent\textit{Claim 2:} For $x \in I_2(T),$ $y^*(x, T)=y_*(x, T)=1$ and $\xi^*(x, T)=\xi_*(x, T)=0.$ Assume otherwise that there exists a point $x_1$ such that $y_*(x_1, T)<1.$ Then we can find a nearby point $x_2$ such that the line segment joining $(x_2, T)$ to $(1, \xi_*(x_2, T))$ intersects the line segment joining $(x_1, T)$ to $(y_*(x_1, T), 0),$ again contradicting the Lemma \ref{unique}.
\end{rem}
Now we are in a position to define $u(x,t)$ and $m(x,t)$ which serve as candidates for the solution to the system \eqref{e1.1}.
\begin{defn}\label{d2}
For $x,t>0$ we define the real-valued function $u(x,t)$ by
\begin{equation*}
u(x,t)=\left\{
\begin{aligned}
&\frac{x-y(x,t)}{t}&&\begin{aligned}&\text{if } F(x,t)<\min\{G_{b_l}(x,t), G_{b_r}(x,t)\} \\
&\text{ and } y_{*}(x,t)=y^{*}(x,t),\end{aligned}\\
&\dfrac{\int_{y_{*}(x,t)}^{y^{*}(x,t)} \rho_0 u_0 }{\int_{y_{*}(x,t)}^ {y^{*}(x,t)} \rho_0} &&\begin{aligned}&\text{if } F(x,t)< \min\{G_{b_l}(x,t),G_{b_r}(x,t)\}\\
&\text{ and }y_{*}(x,t)<y^{*}(x,t),\end{aligned}\\
&\frac{x}{t-\tau(x,t)}&&\begin{aligned}&\text{if } G_{b_l}(x,t) <\min\{F(x,t), G_{b_r}(x,t)\}\\ &\text{ and }\tau_{*}(x,t)=\tau^{*}(x,t),\end{aligned}\\
&\dfrac {\int_{\tau_{*}(x,t)}^{\tau^{*}(x,t)} \rho_b u^2_{b_l}}{\int_{\tau_{*}(x,t)}^ {\tau^{*}(x,t)} \rho_{b_l} u_{b_l}} &&\begin{aligned}&\text{if } G_{b_l}(x,t)<\min\{F(x,t), G_{b_r}(x,t)\}\\
&\text{ and }\tau_{*}(x,t)<\tau^{*}(x,t),\end{aligned}\\
&\frac{x-1}{t-\xi(x,t)}&&\begin{aligned}&\text{if } G_{b_r}(x,t) <\min\{F(x,t), G_{b_l}(x,t)\}\\
&\text{ and }\xi_{*}(x,t)=\xi^{*}(x,t),\end{aligned}\\
&\dfrac {\int_{\xi_{*}(x,t)}^{\xi^{*}(x,t)} \rho_{b_r} u^2_{b_r}}{\int_{\xi_{*}(x,t)}^ {\xi^{*}(x,t)} \rho_{b_r} u_{b_r}} &&\begin{aligned}&\text{if } G_{b_r}(x,t)<\min\{F(x,t), G_{b_l}(x,t)\}\\ &\text{ and }\xi_{*}(x,t)<\xi^{*}(x,t),\end{aligned}\\
&\frac{\int_{0} ^ {\tau^{*}(x,t)}\rho_{b_l} u^2_{b_l}+\int_{0} ^ {y^{*}(x,t)}\rho_0 u_0}{\int_{0} ^ {\tau^{*}(x,t)}\rho_{b_l} u_{b_l} +\int_{0} ^ {y^{*}(x,t)}\rho_0} &&\text{if } F(x,t)=G_{b_l}(x,t)\text{ and  }\begin{aligned} &y^*(x,t)\neq0\text{ or}\\ &\tau^*(x,t)\neq0,\end{aligned}\\
&\frac{x}{t}&&\text{if } F(x,t)=G_{b_l}(x,t) \text{ and }\begin{aligned} &y^*(x,t)=0 \text{ and}\\ &\tau^*(x,t)=0\,,\end{aligned}\\
&\frac{\int_{0} ^ {\xi^{*}(x,t)}\rho_{b_r} u^2_{b_r}+\int_{0} ^ {y_{*}(x,t)}\rho_0 u_0-\int_0^1 \rho_0u_0}{\int_{0} ^ {\xi^{*}(x,t)}\rho_{b_r} u_{b_r} +\int_{0} ^ {y_{*}(x,t)}\rho_0-\int_0^1 \rho_0} &&\text{if } F(x,t)=G_{b_r}(x,t)\text{ and  }\begin{aligned} &y_*(x,t)\neq1\text{ or}\\ &\xi^*(x,t)\neq 0,\end{aligned}\\
&\frac{x-1}{t}&&\text{if } F(x,t)=G_{b_r}(x,t) \text{ and }\begin{aligned} &y_*(x,t)=1 \text{ and}\\ &\xi^*(x,t)=0\,,\end{aligned}\\
&\frac{\int_{0} ^ {\tau^{*}(x,t)}\rho_{b_l} u^2_{b_l}-\int_{0} ^ {\xi^{*}(x,t)}\rho_{b_r} u^2_{b_r}+\int_0^1 \rho_0u_0}{\int_{0} ^ {\tau^{*}(x,t)}\rho_{b_l} u_{b_l} -\int_{0} ^ {\xi^{*}(x,t)}\rho_{b_r} u_{b_r}+\int_0^1 \rho_0} &&\text{if } G_{b_l}(x,t)=G_{b_r}(x,t).
\end{aligned}
\right.
\end{equation*}
For $x=0$ and $t>0,$ we define $u=u_{b_l}$ if $\min\left\{F(0,t), G_{b_r}(0, t)\right\}> G_{b_l}(0,t)$ and $u=0$ if $\min\left\{F(0,t), G_{b_r}(0, t)\right\}$ $<G_{b_l}(0,t)$. For $F(0,t)=G_{b_l}(0,t)$ or $G_{b_l}(0, t)=G_{b_r}(0,t),$ we use the same definition as for $x>0$. Similarly, for $x=1, t>0,$ we define $u=u_{b_r}$ if $\min\left\{F(1, t), G_{b_l}(1, t)\right\}>G_{b_r}(1,t)$ and $u=0$ if $\min\left\{F(1,t), G_{b_l}(1, t)\right\}<G_{b_r}(1,t).$ For $F(1,t)=G_{b_r}(1,t)$ or $G_{b_l}(1, t)=G_{b_r}(1, t),$ we use the same definition as $0<x<1.$
\end{defn}
\begin{defn}\label{d3}
For $t>0$ and $0\leq x\leq 1,$ we define the real valued function $m(x,t)$ by
\begin{equation*}
m(x,t)=\left\{
\begin{aligned}
&\int_{0}^{y_*(x,t)}\rho_0(\eta)\,d\eta&&\begin{aligned}&\text{if }  F(x,t) \leq \min\{G_{b_l}(x,t), G_{b_r}(x,t)\} \\
&\text{ and }0<x<1,\end{aligned}\\
-&\int_{0}^{\tau_*(x,t)}\rho_{b_l}(\eta)u_{b_l}(\eta)\,d\eta&&\begin{aligned}&\text{if }G_{b_l}(x,t)\leq \min\{F(x,t), G_{b_r}(x,t)\}\ \\
&\text{ or }x=0\,,\end{aligned}\\
-&\int_{0}^{\xi_*(x,t)}\rho_{b_r}(\eta)u_{b_r}(\eta)\,d\eta+\int_0^1 \rho_0(\eta)\,d\eta&&\begin{aligned}&\text{if }G_{b_r}(x,t)< \min\{F(x,t), G_{b_l}(x,t)\}\ \\
&\text{ or }x=1\,.\end{aligned}\\
\end{aligned}
\right.
\end{equation*}
\end{defn}
\section{Proof of Theorem \ref{TH1.3}}\label{sec:3} This section aims to prove Theorem \ref{TH1.3}. The proof is divided into two parts: the first establishes the weak formulation (as per Definition \ref{intro-defn-weak formulation}), and the second verifies the initial-boundary conditions and the entropy inequality along the discontinuity of $u(\cdot, t).$
\subsection{Verification of the weak formulation}In this subsection, we prove that the curves defined in Section~\ref{sec:2} form a weak solution of system \eqref{e1.1}. We show that these curves uniquely reach almost all points on the initial and boundary set. Additionally, we show that the minimum of the initial and boundary potentials serves as the space and time anti-derivative of the cumulative distribution of mass and momentum. We begin showing that the constructed curve $X(t)$ can indeed be started from $t=0.$ Since the proof can be completed by closely following the arguments of \cite[Lemma 3.1]{NOSS22}, we omit it here.
\begin{lem}\label{lem:contcurves}
The following holds for the curves defined in the lemma~\ref{lem:curves}.
\begin{enumerate}[label=(\roman*)]
\item \label{defcurvesX}For all but countably many points $(\eta,0)$ on the $x$-axis we have a unique curve $x=X(\eta,t)$, $t\geq0$, such that $X(\eta,0)=\eta$.
\item \label{defcurvesY}For all but countably many points $(0,\eta)$ on the $t$-axis we have a unique curve $x=Y(\eta,t)$, $t\geq\eta$, such that $Y(\eta,\eta)=0$.
\item \label{defcurvesZ}For all but countably many points $(1,\eta)$ on the $t$-axis we have a unique curve $x=Z(\eta,t)$, $t\geq\eta$, such that $Z(\eta,\eta)=1$.
\end{enumerate}
In the countably many exception points define $X(\eta,t)=X(\eta-,t),$ $Y(\eta,t)=Y(\eta-,t)$ and $Z(\eta-, t)=Z(\eta, t)$. Then we have for all $\eta>0$
\begin{equation*}
\begin{split}
\tfrac{\partial}{\partial t}X(\eta, t)&= u(X(\eta, t), t) \quad \textnormal{ for almost all $t>0,$}\\
\tfrac{\partial}{\partial t}Y(\eta, t)&= u(Y(\eta, t), t)  \quad\textnormal{for  almost all $t>\eta,$}\\
\tfrac{\partial}{\partial t}Z(\eta, t)&= u(Z(\eta, t), t)  \quad\textnormal{for  almost all $t>\eta,$}
\end{split}
\end{equation*}
where the right hand side is a measurable function.
\end{lem}
Next we define the momentum $q(x, t)$ and the kinetic energy $E(x, t)$ for different situations.
\begin{defn}\label{def:qE}
Now we define, for $x,t>0$ the momentum and the kinetic energy associated to the system \eqref{e1.1} by
\begin{equation*}\label{e3.1}
q(x,t)=\left\{
 \begin{aligned}
&\int_{0}^{y_*}\rho_0(\eta)u_0(\eta)d\eta, &\text{if } F(x,t)\leq \min\{G_{b_l}(x,t), G_{b_r}(x,t)\},\\
-&\int_{0}^{\tau_*}\rho_{b_l}(\eta)u^2_{b_l} (\eta)d\eta\,,&\text{if } G_{b_l}(x,t)\leq \min\{F(x,t),G_{b_r}(x,t)\}\,,\\
-&\int_{0}^{\xi_*}\rho_{b_r}(\eta)u^2_{b_r} (\eta)d\eta+\int_0^1\rho_0(\eta)u_0(\eta)d\eta\,,&\text{if } G_{b_r}(x,t)< \min\{F(x,t), G_{b_l}(x,t)\}\,.\\
\end{aligned}\right.
\end{equation*}
and
\begin{equation*}\label{e3.2}
E(x,t)=\left\{
\begin{aligned}
&\tfrac{1}{2}\int_{0}^{y_*}\rho_0(\eta)u_0(\eta)u(X(\eta,t),t)d\eta\,,&\text{if }F(x,t)\leq \min\{G_{b_l}(x,t), G_{b_r}(x,t)\},\\
-&\tfrac{1}{2}\int_{0}^{\tau_*}\rho_{b_l}(\eta)u^2_{b_l} (\eta)u(Y(\eta,t),t)d\eta\,, &\text{if }G_{b_l}(x,t)\leq \min\{F(x,t),G_{b_r}(x,t)\}\,,\\
-&\tfrac{1}{2}\int_{0}^{\xi_*}\rho_{b_r}(\eta)u^2_{b_r} (\eta)u(Z(\eta,t),t)d\eta\\
+&\frac{1}{2} \int_0^1 \rho_0(\eta)u_0(\eta)u(X(\eta,t),t)d\eta\,, &\text{if }G_{b_r}(x,t)< \min\{F(x,t), G_{b_l}(x,t)\}\,.\\
\end{aligned}\right.
\end{equation*}
\end{defn}
\begin{lem}\label{l51}
In the sense of Radon-Nikodym derivatives in $x$, the following holds in the interior of $\Omega$:
(i) $dq=udm$,
(ii) $dE=\frac{1}{2}u^2dm$.
\end{lem}
\begin{proof}
If $F(x,t)< \min\{G_{b_l}(x,t), G_{b_r}(x,t)\}$, the result follows from Lemma~2.8. in \cite{Wang01} and similarly, if $G_{b_l}(x, t)< \min\left\{F(x, t), G_{b_r}(x, t)\right\},$ then the result is Lemma 3.4 in \cite{NOSS22}. Now, we consider a point $(x, t)$ where $G_{b_r}(x,t)<\min\{F(x,t), G_{b_l}(x,t)\}$.
If $\xi_*$ is constant in some neighborhood $(x,t)$, then the above quantities are constant  and lemma holds trivially. Now suppose $\xi_*(x,t)$ is not constant in a neighborhood of $(x,t)$ and assume
$\xi_*(x,t)= \xi^ *(x,t)=\xi(x,t)$. Let $x_1 < x < x_2$, then by definition
\begin{equation*}
\begin{aligned}
G_{b_r}(\xi_{*}(x_1,t), x_1, t)= \int_{0} ^ {\xi_{*}(x_1,t)}[x_1-1-(t-\eta)u_{b_r}(\eta)]\rho_{b_r}(\eta)u_{b_r}(\eta)d\eta+ F(1, x_1, t)\\
G_{b_r}(\xi_{*}(x_2,t), x_1 ,t)= \int_{0} ^ {\xi_{*}(x_2,t)}[x_1-1-(t-\eta)u_{b_r}(\eta)]\rho_{b_r}(\eta)u_{b_r}(\eta)d\eta + F(1, x_1, t)\,.
\end{aligned}
\end{equation*}
By the minimizing properties we have
\begin{equation*}
G_{b_r}(\xi_{*}(x_1,t), x_1, t) \leq G_{b_r}(\xi_{*}(x_2,t), x_1 ,t)\,,
\end{equation*}
and the definitions lead us to the following inequality:
\begin{equation}\label{e3.5}
\frac{x_1}{t-\xi_{*}(x_2,t)} \geq \frac{\int_{\xi_{*}(x_1,t)}^{\xi_{*}(x_2,t)}\big(\frac{t-\eta}{t-\xi_{*}(x_2,t)}\big)u^2_{b_r}(\eta)\rho_{b_r}(\eta)d\eta}{\int_{\xi_{*}(x_1,t)}^{\xi_{*}(x_2,t)}\rho_{b_r}(\eta)u_{b_r}(\eta)d\eta}\,.
\end{equation}
Now since
\[
\frac{t-\eta}{t-\xi_*(x_2,t)}\geq 1\,,
\]
and since $\xi_*$ is semi-continuous we can take the limits $x_1\nearrow x$ and $x_2\searrow x$ in \eqref{e3.5} to derive
\begin{equation}\label{e3.6}
\frac{x}{t-\xi(x,t)}\geq \lim_{x_2, x_1 \to x}\frac{q(x_2,t)-q(x_1,t)}{m(x_2,t)-m(x_1,t)}\,.
\end{equation}
Similarly considering the inequality
\begin{equation*}
G_{b_r}(\xi_{*}(x_2,t), x_2, t) \leq G_{b_r}(\xi_{*}(x_1,t), x_2 ,t)
\end{equation*}
and following the analysis as above we get
\begin{equation}
\frac{x}{t-\tau(x,t)}\leq \lim_{x_2, x_1 \to x}\frac{q(x_2,t)-q(x_1,t)}{m(x_2,t)-m(x_1,t)}.
\label{e3.7}
\end{equation}
From equations \eqref{e3.6} and \eqref{e3.7}, we conclude $dq= u dm$.\\
If  $\xi_{*}(x,t)< \xi^{*}(x,t)$, then
\begin{equation*}
\lim_{x_2, x_1 \to x}\frac{q(x_2,t)-q(x_1,t)}{m(x_2,t)-m(x_1,t)}
=\lim_{x_2, x_1 \to x}\frac{\int_{\xi_{*}(x_2,t)} ^ {\xi_{*}(x_1,t)}\rho_{b_r} u_{b_r}^2}{\int_{\xi_{*}(x_2,t)} ^ {\xi_{*}(x_1,t)}\rho_{b_r} u_{b_r}}
= \frac{\int_{\xi_{*}(x,t)} ^ {\xi^{*}(x,t)}\rho_{b_r} u^2_{b_r}}{\int_{\xi_{*}(x,t)} ^ {\xi^{*}(x, t)}\rho_{b_r} u_{b_r}}\,.
\end{equation*}
Now we consider the remaining cases, where $(x, t)$ is a point with $F(x,t)= G_{b_l}(x,t),$ or $F(x,t)=G_{b_r}(x,t)$ or $G_{b_l}(x,t)=G_{b_r}(x,t)$. The first case is considered in \cite{NOSS22}. Next, let $F(x,t)=G_{b_r}(x,t)$ holds for isolated point. Then for $x_1<x<x_2,$ we have  
\begin{multline*}\label{e3.9}
\begin{aligned}
\lim_{x_2, x_1 \to x}\frac{q(x_2,t)-q(x_1,t)}{m(x_2,t)-m(x_1,t)}&=\lim_{x_2, x_1 \to x}\frac{-\int_{0} ^ {\xi_{*}(x_2,t)}\rho_{b_r} u^2_{b_r}-\int_{0} ^ {y_{*}(x_1,t)}\rho_0 u_0+\int_0^1\rho_0 u_0}{-\int_{0} ^ {\xi_{*}(x_2,t)}\rho_{b_r} u_{b_r} -\int_{0} ^ {y_{*}(x_1,t)}\rho_0+\int_0^1 \rho_0}\\
&= \frac{\int_{0} ^ {\xi^{*}(x,t)}\rho_{b_r} u^2_{b_r}+\int_{0} ^ {y^{*}(x,t)}\rho_0 u_0-\int_0^1 \rho_0 u_0}{\int_{0} ^ {\xi^{*}(x,t)}\rho_{b_r} u_{b_r} +\int_{0} ^ {y^{*}(x,t)}\rho_0-\int_0^1 \rho_0}\,=u(x,t).
\end{aligned}
\end{multline*}
Here we used $\xi_*(x+, t)=\xi^*(x,t)$ and $y_*(x-, t)=y_*(x,t)$. Similarly using the property $\xi_*(x_2,t) \to \xi^*(x,t)$ and $\tau_*(x,t) \to \tau^*(x,t)$ one can show
\[\lim_{x_2, x_1 \to x}\frac{q(x_2,t)-q(x_1,t)}{m(x_2,t)-m(x_1,t)}=u(x,t)\] when $G_{b_l}(x,t)=G_{b_r}(x,t)$ holds for isolated points.
On the other hand, if $F(x,t)=G_{b_l}(x,t)$ holds in an entire neighborhood of $(x,t)$ then the region corresponds to a rarefaction wave originating from zero. Consequently, $\tau^*(x,t)=y^*(x,t)=0$ throughout this neighborhood. By definition $m$, $q$ and $E$ are all zero, completing the proof. Similarly, if $F(x, t)=G_{b_r}(x, t)$ in a neighborhood of $(x, t),$ then this case corresponds to the rarefaction wave emanating from $(1, 0)$ and hence by Remark \ref{rem2.9} $y^*(x, t)=1, \xi^*(x,t)=0$ throughout the neighborhood. From the definition, we have $m(x, t)=\int_0^1 \rho_0(\eta)\, d\eta, q(x, t)=\int_0^1 \rho_0(\eta)u_0(\eta)\,d\eta$ and $E(x, t)=\frac{1}{2}\int_0^1 \rho_0(\eta)u_0(\eta)u(X(\eta, t), t)\,d\eta.$ Thus $dq=dm=dE=0$ and thus the proof is finished. Moreover, from Remark \ref{rem2.9}, we have that $G_{b_l}(x, t)=G_{b_r}(x, t)$ can be possible only at the isolated points.

In the boundary points of the region $F(x,t)=G_{b_l}(x,t)$ or $F(x,t)=G_{b_r}(x, t),$ the same proof as $F<G_{b_l}, F>G_{b_l}$ or $F<G_{b_r}, F>G_{b_r}$ works on the left and right boundary, respectively. Thus in all possible cases we derived that $dq= u\, dm$ in the sense of Radon-Nikodym derivative.

Now we turn our attention to the proof of $dE=\frac{1}{2} u^2 dm$.
First, let $(x,t)$ again be a point where $G_{b_r}(x,t) < \min\{F(x,t), G_{b_l}(x,t)\}$ and $\xi_{*}(x,t)=\xi^{*}(x,t)$. Then
\begin{equation}\label{e3.10}
E(x_1, t)-E(x_2, t)= \tfrac{1}{2}\int_{\xi_{*}(x_1, t)} ^ {\xi_{*}(x_2, t)} \rho_{b_r} (\eta) u^2_{b_r} (\eta) u(Z(\eta, t), t)d\eta\,.
\end{equation}
Note that for  $\xi_{*}(x_1, t) \leq \eta \leq \xi_{*}(x_2, t)$ we have
\begin{equation}\label{e3.11}
\frac{x-1}{t- \xi_{*}(x_1, t)} \leq u(Z(\eta, t), t) \leq  \frac{x-1}{t- \xi_{*}(x_2, t)}\,.
\end{equation}
Hence from equation \eqref{e3.10} and \eqref{e3.11} we derive
\[
\frac{1}{2}\frac{x-1}{t- \xi_{*}(x_1, t)} \leq \frac{E(x_2, t)-E(x_1, t)}{q(x_1, t)-q(x_2, t)} \leq \frac{1}{2} \frac{x-1}{t- \xi_{*}(x_2, t)}\,.
\]
In the limit $x_1\nearrow x\swarrow x_2$, we have $\xi_*(x_1,t) \to \xi_*(x,t)$ and $\xi_*(x_2,t) \to  \xi^*(x,t),$ hence $\frac{E(x_2, t)-E(x_1, t)}{q(x_1, t)-q(x_2, t)} \to \frac{1}{2}u$ and
we know $\frac{dq}{dm}= u$. Combining these two, we get $dE= \frac{1}{2} u^2 dm$.\\
For the case $F(x,t)=G_{b_r}(x,t)$ in an isolated point. Then noting that $(x_1,t)$ lies left of $(x,t)$ and thus $F(x_1,t)<G_{b_r}(x_1,t)$, and the opposite holds true for $(x_2,t).$ Thus, we have
\begin{equation*}
\begin{split}
&\lim_{x_1, x_2 \to x}\frac{E(x_2, t)-E(x_1, t)}{q(x_2, t)-q(x_1, t)}=\\
&= \frac{\frac{1}{2}\int_{0}^{y_* (x,t)}\rho_0(\eta)u_0(\eta)u(X(\eta,t),t)d\eta+\frac{1}{2}\int_{0}^{\xi^{*} (x,t)}\rho_{b_r}(\eta)u^2_{b_r} (\eta)u(Z(\eta,t),t)d\eta}
{\int_{0}^{y_* (x,t)}\rho_0(\eta)u_0(\eta)d\eta+\int_{0}^{\xi^{*} (x,t)}\rho_{b_r}(\eta)u^2_{b_r} (\eta) d\eta-\int_0^1\rho_0(\eta)u_0(\eta)d\eta}\\
&-\frac{-\frac{1}{2}\int_0^1 \rho_0(\eta)u_0(\eta)u(X(\eta,t),t)d\eta}{\int_{0}^{y_* (x,t)}\rho_0(\eta)u_0(\eta)d\eta+\int_{0}^{\xi^{*} (x,t)}\rho_{b_r}(\eta)u^2_{b_r} (\eta) d\eta-\int_0^1\rho_0(\eta)u_0(\eta)d\eta}.
\end{split}
\end{equation*}
For $\eta \in [y_* (x,t)), 1]$, we have $u(X(\eta,t),t)=u(x,t)$ and similarly for  $\eta \in [0, \xi^ * (x,t))]$ we know that $u(Z(\eta,t),t)=u(x,t)$. Also we notice that for $\eta \in [0,1],$ $u(X(\eta,t),t)=u(x,t).$\\
Thus from the equation above we have
\[
\lim_{x_1, x_2 \to x}\frac{E(x_2, t)-E(x_1, t)}{q(x_2, t)-q(x_1, t)}=\frac{1}{2}u(x,t)\,.
\]
Since $dq =u dm$,  we again get  $dE=\frac{1}{2} u^2dm$.
This completes the proof.
\end{proof}
\begin{lem} \label{lem:integrals}
Define
\[
\mu(x,t)= \min \left\{F(x,t), G_{b_l}(x,t), G_{b_r}(x, t)\right\}\,,
\]
then the following  holds for $x_1,x_2,t>0$:
\begin{equation}
\int_{x_1}^{x_2} m(x,t)dx= \mu(x_1,t)-\mu(x_2, t)\,.
\label{e3.13}
\end{equation}
\begin{equation}
\int_{t_1}^{t_2} q(x,t) dt =\mu(x, t_2)- \mu (x, t_1)\,.
\label{e3.14}
\end{equation}
\end{lem}

\begin{proof}
Let $t>0$ be fixed. We start with the proof of \eqref{e3.13}. Now for any two points $x, x^{\prime} \in [x_1, x_2],$  with $ x< x^{\prime}$,  we claim the following:
\begin{equation}\label{e3.21}
(x-x^{\prime}) m (x^{\prime}, t) \leq \mu(x^{\prime},t)-\mu(x, t)\leq (x-x^{\prime}) m (x, t).
\end{equation}
Depending upon the minimization, we have the possible values of $\mu(x, t)$ and $\mu(x^{\prime}, t):$
\begin{equation}\label{case}
\begin{aligned}
&(a)\,\,\,\, \mu(x,t)=F(x,t),\, \mu(x^{\prime},t)=F(x^{\prime},t),\,\, &&(b)\,\,\,\, \mu(x,t)=G_{b_r}(x,t), \,\mu(x^{\prime},t)=F(x^{\prime},t),\\
&(c)\,\,\,\, \mu(x,t)=G_{b_l}(x,t),\, \mu(x^{\prime},t)=G_{b_l}(x^{\prime},t),\,\,&&(d)\,\,\,\, \mu(x,t)=G_{b_r}(x,t),\, \mu(x^{\prime},t)=G_{b_r}(x^{\prime},t),\\
&(e)\,\,\,\, \mu(x,t)=G_{b_l}(x,t),\, \mu(x^{\prime},t)=F(x^{\prime},t),\,\,&&(f)\,\,\,\, \mu(x,t)=G_{b_l}(x,t),\, \mu(x^{\prime},t)=G_{b_r}(x^{\prime},t).
\end{aligned}
\end{equation}
The proof of the inequality \eqref{e3.21} for the case of (a) (c) and (e) in \eqref{case} can be found in \cite[Lemma~3.5]{NOSS22}. Here we prove for the cases (b), (d) and (f).  In case (b) we have
\begin{equation*}\
\begin{aligned}
&\mu(x^{\prime},t)-\mu(x,t)\\
&=F(y_*(x^{\prime},t), x^{\prime}, t)-G_{b_r}(\xi_*(x,t), x, t)\\
&=[F(y_*(x^{\prime},t), x^{\prime}, t)-F(y_*(x^{\prime},t), x, t)]+[F(y_*(x^{\prime},t), x, t)-G_{b_r}(\xi_*(x,t), x, t)]
\end{aligned}
\end{equation*}
Since $F(y_*(x^{\prime},t), x, t)-G_{b_r}(\xi_*(x,t), x, t) \geq 0$, we get
\begin{equation}\label{eq3.17}
\mu(x^{\prime},t)-\mu(x,t)\geq F(y_*(x^{\prime},t), x^{\prime}, t)-F(y_*(x^{\prime},t), x, t).
\end{equation}
On the other hand, 
\begin{equation*}
\begin{aligned}
&\mu(x^{\prime},t)-\mu(x,t)\\
&=F(y_*(x^{\prime},t), x^{\prime}, t)-G_{b_r}(\xi_*(x,t), x, t)\\
&=[F(y_*(x^{\prime},t), x^{\prime}, t)-G_{b_r}(\xi_*(x,t), x^{\prime}, t)]+[G_{b_r}(\xi_*(x,t), x^{\prime}, t)-G_{b_r}(\xi_*(x,t), x, t)].
\end{aligned}
\end{equation*}
Since $F(y_*(x^{\prime},t), x^{\prime}, t)-G_{b_r}(\xi_*(x,t), x^{\prime}, t)\leq 0$, we get
\begin{equation}\label{eq3.18}
\mu(x^{\prime},t)-\mu(x,t)\leq G_{b_r}(\xi_*(x,t), x^{\prime}, t)-G_{b_r}(\xi_*(x,t), x, t).
\end{equation}
Again, from the equations \eqref{eq3.17}-\eqref{eq3.18} and using the definition of $m$, $F$ and $G_{b_r}$, we get \eqref{e3.21}.

\noindent Now for the case (d), we have
\begin{multline*}
\mu(x^{\prime},t)-\mu(x, t)=G_{b_r}(\xi_{*}(x^{\prime},t), x^{\prime}, t)-G_{b_r}(\xi_{*}(x,t), x, t)\\
=[G_{b_r}(\xi_{*}(x^{\prime},t), x^{\prime}, t)-G_{b_r}(\xi_{*}(x,t), x^{\prime}, t)]+[G_{b_r}(\xi_{*}(x,t), x^{\prime}, t)-G_{b_r}(\xi_{*}(x,t), x, t))]\,.
\end{multline*}
Since the first bracket of the above expression is negative, we get
\begin{equation} \label{e3.17}
 \mu(x^{\prime},t)-\mu(x, t) \leq G_{b_r}(\xi_{*}(x,t), x^{\prime}, t)-G_{b_r}(\xi_{*}(x,t), x, t)\,.
\end{equation}
On the other hand, we write
 \begin{multline*}
\mu(x^{\prime},t)-\mu(x, t)=G_{b_r}(\xi_{*}(x^{\prime},t), x^{\prime}, t)-G_{b_r}(\xi_{*}(x,t),  x, t)\\
= [G_{b_r}(\xi_{*}(x^{\prime},t), x^{\prime}, t)-G_{b_r}(\xi_{*}(x^{\prime},t), x , t)]+[G_{b_r}(\xi_{*}(x^{\prime},t), x, t)-G_{b_r}(\xi_{*}(x,t), x, t))]\,.
\end{multline*}
Now the second bracket of the expression is positive and thus
\begin{equation}\label{e3.19}
\mu(x^{\prime},t)-\mu(x, t) \geq G_{b_r}(\xi_{*}(x',t), x^{\prime}, t)-G_{b_r}(\xi_{*}(x',t), x, t)\,.
\end{equation}
Combining inequalities \eqref{e3.17}, \eqref{e3.19} and using the definitions of $G_{b_r}$ and $m$ we have \eqref{e3.21}. Following the similar argument we can prove (f).

Now since $m$ is monotonous in $x$, it is also Riemann integrable. Taking Riemann sums and using \eqref{e3.21}, we get
\[
\int_{x_1}^{x_2} m(x,t)dx= \mu(x_1,t)-\mu(x_2, t)\,.
\]
To prove \eqref{e3.14} we claim the following inequality holds:
\begin{equation}\label{eq3.20}
(t'-t)q(x,t')\leq\mu(x,t')-\mu(x,t)\leq(t'-t)q(x,t)\,.
\end{equation}
 To prove the above inequality, we have the following possibilities for $\mu(x,t)$ and $\mu(x,t^{\prime})$:
\begin{equation}\label{Equation 3.18}
\begin{aligned}
&(a)\,\,\,\, \mu(x,t)=F(x,t),\,\, \mu(x,t^{\prime})=F(x,t^{\prime}),&&(b)\,\,\,\, \mu(x,t)=G_{b_l}(x,t),\,\, \mu(x,t^{\prime})=G_{b_l}(x,t^{\prime}),\\
&(c)\,\,\,\, \mu(x,t)=F(x,t),\,\, \mu(x,t^{\prime})=G_{b_l}(x,t^{\prime}),&&(d)\,\,\,\, \mu(x,t)=G_{b_l}(x,t),\,\, \mu(x,t^{\prime})=F(x,t^{\prime})\\
&(e)\,\,\,\, \mu(x,t)=F(x,t),\,\, \mu(x,t^{\prime})=G_{b_r}(x,t^{\prime}), &&(f)\,\,\,\, \mu(x,t)=G_{b_l}(x,t),\,\, \mu(x,t^{\prime})=G_{b_r}(x,t^{\prime}),\\
&(g)\,\,\,\, \mu(x,t)=G_{b_r}(x,t),\,\, \mu(x,t^{\prime})=G_{b_l}(x,t^{\prime}),&&(h)\,\,\,\, \mu(x,t)=G_{b_r}(x,t),\,\, \mu(x,t^{\prime})=F(x,t^{\prime})\\
&(i)\,\,\,\, \mu(x,t)=G_{b_r}(x,t),\,\, \mu(x,t^{\prime})=G_{b_r}(x,t^{\prime}).
\end{aligned}
\end{equation}
The proof for the cases (a)-(d) on $\{x\}\times[t_1,t_2]$ can be found in \cite[Lemma 3.5]{NOSS22}. Here we shall prove the cases (e)-(i). Since the proofs are similar, we only present the proof of case $(i)$ here. For $t'>t\in[t_1,t_2],$ we observe,
\begin{align}\label{ee3.22}
\mu(x,t^{\prime})-\mu(x, t)&=G_{b_r}(\xi_{*}(x,t^{\prime}), x, t^{\prime})-G_{b_r}(\xi_{*}(x,t), x, t)\nonumber\\
&= [G_{b_r}(\xi_{*}(x,t^{\prime}), x, t^{\prime})-G_{b_r}(\xi_{*}(x,t), x, t^{\prime})]+[G_{b_r}(\xi_{*}(x,t), x, t^{\prime})-G_{b_r}(\xi_{*}(x,t), x, t))]\nonumber\\
&\leq G_{b_r}(\xi_{*}(x,t), x, t^{\prime})-G_{b_r}(\xi_{*}(x,t), x, t))\,,
\end{align}
and
\begin{align}\label{ee3.23}
\mu(x,t^{\prime})-\mu(x, t)&=G_{b_r}(\xi_{*}(x,t^{\prime}), x, t^{\prime})-G_{b_r}(\xi_{*}(x,t), x, t)\nonumber\\
&=[G_{b_r}(\xi_{*}(x,t^{\prime}), x, t^{\prime})-G_{b_r}(\xi_{*}(x,t^{\prime}), x, t)]+[G_{b_r}(\xi_{*}(x,t^{\prime}), x, t)-G_{b_r}(\xi_{*}(x,t), x, t))]\nonumber\\
&\geq G_{b_r}(\xi_{*}(x,t^{\prime}), x, t^{\prime})-G_{b_r}(\xi_{*}(x,t^{\prime}), x, t)\,.
\end{align}
Combining the inequalities \eqref{ee3.22}-\eqref{ee3.23} and using the definition of $q(x, t),$ we get \eqref{eq3.20}. Since for a fixed $x$, $y_*(x,t)$ is monotone in the interval $[t_1,t_2]$, $q(x,t)$ is a function of bounded variation and hence Riemann integrable. Thus following a similar argument as before, identity \eqref{e3.14} follows from \eqref{eq3.20}.
\end{proof}
\begin{rem}
We observe that \eqref{case} presents fewer possibilities compared to \eqref{Equation 3.18}. The excluded possibilities in \eqref{case} are ruled out by Lemma \ref{unique}, as they would lead to a contradiction. 
\end{rem}
From the above Lemma \ref{lem:integrals}, we have $\mu_x= m$ and $\mu_t = q.$ Let $\varphi$ is a test function and we show the first integral identity \eqref{weak formulation} of Definition \ref{intro-defn-weak formulation} holds.
\begin{align}\label{e3.22}
0=\int \int [\mu_x \phi_t (x, t)- \mu_t \phi_x (x, t)] dx dt
=\int \int [m(x,t) \phi_t (x, t)- q(x,t)\phi_x (x, t)] dx dt=\nonumber\\
=\int \int [m(x,t) \phi_t (x, t)dx dt- \int \int u(x,t)\phi(x, t)dm dt\,.
\end{align}
This identity proves that $(\rho, u)$ satisfies the first equation of \eqref{e1.1} in the sense of Definition \ref{intro-defn-weak formulation}.
To prove the second equation, the following lemma is crucial.
\begin{lem}\label{lem: h function}
For $x, t>0$,  let us define
 \begin{equation*}
 H(x, t)=\left\{
 \begin{aligned}
 &H_1(x,t)=\int_0 ^ {y_{*} (x,t)} \rho_0 (\eta) u_0 (\eta) (X(\eta, t)-x) d\eta\,, &&\text{if }  F(x,t)\leq \min\{ G_{b_l}(x,t), G_{b_r}(x,t)\},\\
&H_2(x,t)=-\int_0 ^ {\tau_{*}(x,t)} \rho_b (\eta) u^2_b (\eta) (Y(\eta, t)-x) d\eta\,, &&\text{if } G_{b_l}(x,t)\leq \min \{F(x,t), G_{b_r}(x,t)\},\\
 &H_3(x,t)=-\int_0 ^ {\xi_{*}(x,t)} \rho_{b_r} (\eta) u^2_{b_r} (\eta) (Z(\eta, t)-x) d\eta\\
&+\int_0^1\rho_0(\eta)u_0(\eta)(X(\eta,t)-x)d\eta\,, &&\text{if } G_{b_r}(x,t)< \min \{F(x,t), G_{b_l}(x,t)\},
\end{aligned}\right.
 \end{equation*}
then we have $H_x = - q $ and $H_t = 2 E$.
\end{lem}
\begin{proof}
We first show,
\begin{equation}\label{e3.24}
-\int_{x_1}^{x_2} q(x,t)dx= H(x_1,t)-H(x_2, t)\,.
\end{equation}
Let $t$ be fixed and $[x_1, x_2]$ be an interval. Assume that $x$ and $x^{\prime}$ be any two points in $[x_1,x_2]$ with $x<x^\prime.$ To prove the above identity it is enough to prove
\begin{equation}\label{e3.29}
 -(x- x^{\prime}) q(x, t)\leq H(x, t)-H(x^{\prime}, t) \leq -(x- x^{\prime}) q(x^{\prime}, t)\,.
\end{equation}
First, we consider the case where the potentials are not equal at $x$ or $x^{\prime}.$ Now, depending upon the minimizations through $F$, $G_{b_l}$ and $G_{b_r}$, one has the following possibilities.
\begin{equation*}
\begin{aligned}
&(a)\,\,H(x,t)=H_1(x,t),\,\, H(x^\prime,t)=H_1(x^\prime,t), &&(b)\,\,H(x,t)=H_2(x,t),\,\, H(x^\prime,t)=H_2(x^\prime,t),\\
&(c)\,\,H(x,t)=H_2(x,t),\,\, H(x^\prime,t)=H_1(x^\prime,t), &&(d)\,\,H(x,t)=H_2(x,t),\,\, H(x^\prime,t)=H_3(x^\prime,t),\\
&(e)\,\,H(x,t)=H_1(x,t),\,\, H(x^\prime,t)=H_3(x^\prime,t), &&(f)\,\,H(x,t)=H_3(x,t),\,\, H(x^\prime,t)=H_3(x^\prime,t).
\end{aligned}
\end{equation*}
For case $(a), (b), (c)$ the result holds true by arguments in \cite{WHD97} and \cite{NOSS22}. The proof of $(d),(e), (f),$ are similar to the proof of Lemma~\ref{lem:integrals}.

\noindent Let us define $$H_3(\xi, x, t)=-\int_0 ^ {\xi} \rho_{b_r} (\eta) u^2_{b_r} (\eta) (Z(\eta, t)-x) d\eta+ \int_0^1\rho_0(\eta)u_0(\eta)(Z(\eta,t)-x)d\eta.$$ Since $Z(\eta, t)$ is increasing in $\eta$  and $Z(\eta, t)=x$ for $\xi_{*}(x, t)< \eta< \xi^{*}(x, t)$, we have
\begin{equation}\label{e3.26}
H_3 (x,t)= \min_{\xi\geq 0} H_3 (\xi, x, t).
\end{equation}
From equation \eqref{e3.26}, it is evident that $H_3(x,t)$ is a continuous function. First, we present a proof of $(d).$ For $x$ and $x^{\prime}$ as defined above, we have
\begin{align*}
H(x, t)-H(x^{\prime}, t)&= H_2(x, t)-H_3(x^{\prime}, t)= H_2 (\tau_{*}(x,t), x, t)- H_3(\xi_{*}(x^{\prime},t),x^{\prime}, t)\\
&=[H_2 (\tau_{*}(x,t), x, t)-H_3 (\xi_{*}(x^{\prime},t), x, t)]+ [H_3 (\xi_{*}(x^{\prime},t), x, t)-H_3(\xi_{*}(x^{\prime},t),x^{\prime}, t)]\\
&\leq H_3 (\xi_{*}(x^{\prime},t), x, t)-H_3(\xi_{*}(x^{\prime},t),x^{\prime}, t) \,.
\end{align*}
On the other hand
\begin{align*}
H(x, t)-H(x^{\prime}, t)&= H_2(x, t)-H_3(x^{\prime}, t)= H_2 (\tau_{*}(x,t), x, t)- H_3(\xi_{*}(x^{\prime},t),x^{\prime}, t)\\
&=[H_2 (\tau_{*}(x,t), x, t)-H_2 (\tau_{*}(x,t), x^{\prime}, t)]+ [H_2 (\tau_{*}(x,t), x^{\prime}, t)-H_3(\xi_{*}(x^{\prime},t),x^{\prime}, t)]\\
&\geq H_2 (\tau_{*}(x,t), x, t)-H_2 (\tau_{*}(x,t), x^{\prime}, t) \,.
\end{align*}
Combining those two inequalities, one concludes the inequality in \eqref{e3.29}. Using \eqref{e3.29}, and taking the supremum of the Riemann sums over all partitions of the interval $[x_1, x_2]$, we deduce
\eqref{e3.24}. In the case of $(e),$ similarly, we can write
\begin{equation*}
\begin{aligned}
H(x,t)-H(x^{\prime},t)&=H_1(x,t)-H_3(x^{\prime},t)=H_1(y_*(x,t), x, t)-H_3(\xi_*(x^{\prime},t),x^{\prime}, t)\\
&=[H_1(y_*(x,t),x,t)-H_1(y_*(x,t),x^{\prime},t)]+[H_1(y_*(x,t), x^{\prime},t)-H_3(\xi_*(x^{\prime},t),x^{\prime},t)]
\end{aligned}
\end{equation*}
Since the term in the second bracket is positive, we conclude
\begin{equation}\label{case a+}
H(x,t)-H(x^{\prime},t)\geq H_1(y_*(x,t),x,t)-H_1(y_*(x,t),x^{\prime},t).
\end{equation}
On the other hand, we have
\begin{equation*}
\begin{aligned}
&H(x,t)-H(x^{\prime},t)=H_1(x,t)-H_3(x^{\prime},t)=H_1(y_*(x,t), x, t)-H_3(\xi_*(x^{\prime},t),x^{\prime}, t)\\
&=[H_1(y_*(x,t),x,t)-H_3(\xi_*(x^{\prime},t),x,t)]+[H_3(\xi_*(x^{\prime},t), x,t)-H_3(\xi_*(x^{\prime},t),x^{\prime},t)].
\end{aligned}
\end{equation*}
Since the term in the first bracket is negative we get
\begin{equation}\label{case a-}
H(x,t)-H(x^{\prime},t)\leq H_3(\xi_*(x^{\prime},t), x,t)-H_3(\xi_*(x^{\prime},t),x^{\prime},t).
\end{equation}
Combining the inequalities \eqref{case a+}-\eqref{case a-} and the definition of $q$, we get \eqref{e3.29}.
Now following the same argument as above we obtain \eqref{e3.24}.

Now we consider the delicate case where $F(x, t)=G_{b_l}(x, t)$ and $F(x^{\prime}, t)=G_{b_r}(x^{\prime}, t).$ Due to Lemma \ref{unique}, this implies $H(x_1, t)=H_2(x_1, t)$ and $H(x_2, t)=H_3(x_2, t).$ We observe that
\begin{align}\label{EQQ3.28}
    -\int_{x_1}^{x_2} q(x, t)\,dx&=-\int_{[x_1, x)} q(x, t)\,dx-\int_{[x, x^{\prime}]} q(x, t)\,dx-\int_{(x^{\prime}, x_2]} q(x, t)\,dx\nonumber\\
    &=-[H_2(x-, t)-H_2(x_1, t)]-[H_1(x^{\prime}, t)-H_1(x, t)]-[H_3(x_2, t)-H_3(x^{\prime}+, t)].
\end{align}
Note that, since $X(\eta, t)=x,$ for  $\eta \in [0, y^*(x, t)],$  it follows $H_1(x, t)=0.$
Also, since $Z(\eta, t)=x^{\prime}$ for $\eta \in [0, \xi^*(x^{\prime}, t)]$ and $X(\eta, t)=x^{\prime}$ for $\eta \in [y_*(x^{\prime}, t), 1],$ using $y_*(x^{\prime}+, t)=y^*(x^{\prime}, t),$ and $\xi_*(x^{\prime}+, t)=\xi^*(x^{\prime}, t),$ we get
\begin{align*}
    H_3(x^{\prime}+, t)&=\int_0^{\xi_*(x^{\prime}+, t)}\rho_{b_r}(\eta)u^2_{b_r}(\eta)(Z(\eta, t)-x^{\prime})\,d\eta+ \int_0^1\rho_0(\eta)u_0(\eta)(X(\eta, t)-x^{\prime})\,d\eta\\
    &=\int_0^{\xi_*(x^{\prime}+, t)}\rho_{b_r}(\eta)u^2_{b_r}(\eta)(Z(\eta, t)-x^{\prime})\,d\eta+ \int_0^{y_*(x^{\prime}+, t)}\rho_0(\eta)u_0(\eta)(X(\eta, t)-x^{\prime})\,d\eta\\
    &+ \int_{y_*(x^{\prime}+, t)}^1\rho_0(\eta)u_0(\eta)(X(\eta, t)-x^{\prime})\,d\eta\\
    &=\int_0^{\xi^*(x^{\prime}, t)}\rho_{b_r}(\eta)u^2_{b_r}(\eta)(Z(\eta, t)-x^{\prime})\,d\eta+\int_0^{y^*(x^{\prime}, t)}\rho_0(\eta)u_0(\eta)(X(\eta, t)-x^{\prime})\,d\eta\\
    &+\int_{y^*(x^{\prime}, t)}^1\rho_0(\eta)u_0(\eta)(X(\eta, t)-x^{\prime})\,d\eta=\int_0^{y^*(x^{\prime}, t)}\rho_0(\eta)u_0(\eta)(X(\eta, t)-x^{\prime})\,d\eta.
\end{align*}
Moreover, using $\tau_{*}(x-, t)=\tau^*(x, t)$ and $Y(\eta, t)=x$ for $\eta \in [0, \tau^*(x, t)],$ we have $H_2(x-, t)=0.$ Also, we note that
\begin{align*}
    H_1(x^{\prime}, t)-H_3(x^{\prime}+, t)=-\int_{y_*(x^{\prime}, t)}^{y^*(x^{\prime}, t)}\rho_0(\eta) u_0(\eta)(X(\eta, t)-x^{\prime})\,d\eta=0.
\end{align*}
Thus, from \eqref{EQQ3.28}, we obtain
\begin{align*}
    -\int_{x_1}^{x_2}q(x, t)\, dx=-[H_3(x_2,t)-H_2(x_1, t)]=H(x_1, t)-H(x_2, t).
\end{align*}
Next we consider the more complicated case where $F(x, t)=G_{b_l}(x, t)$ in an interval $[x^{\prime}_1, x]\subset [x_1, x_2]$ and $F(x, t)=G_{b_r}(x, t)$ in an interval $[x^{\prime}_2, x^{\prime}] \subset [x_1, x_2],$ with $x\leq x^{\prime}_2.$ Then, we have
\begin{align*}
    &-\int_{x_1}^{x_2}q(x, t)\,dx=-\int_{[x_1, x^{\prime}_1)}q(x, t)\,dx-\int_{[x^{\prime}_1, x]}q(x, t)\,dx-\int_{(x, x^{\prime}_2]}q(x, t)\,dx\\
    &-\int_{(x^{\prime}_2, x^{\prime}]}q(x, t)\,dx-\int_{(x^{\prime},x_2]}q(x, t)\,dx=[H_2(x_1, t)-H_2(x^{\prime}_1-, t)]+0+[H_1(x+, t)-H_3(x^{\prime}_2, t)]\\
    &+[H_1(x^{\prime}_2+, t)-H_3(x^{\prime}, t)]+[H_3(x^{\prime}+, t)-H_3(x_2, t)].
\end{align*}
By definition, $H_2(x^{\prime}_1-, t)=0=H_1(x+, t), $ $H_3(x^{\prime}_2, t)=H_1(x^{\prime}_2+, t)=\int_0^1 \rho_0 u_0 (X(\eta, t)-x^{\prime}_2)\,d\eta,$  $H_3(x^{\prime}+, t)=H_3(x^{\prime}, t)=\int_0^1\rho_0 u_0 (X(\eta, t)-x^{\prime})\,d\eta.$

Thus, we conclude
\begin{align*}
 -\int_{x_1}^{x_2}q(x, t)\,dx=H_2(x_1, t)-H_3(x_2, t)=H(x_1, t)-H(x_2, t).   
\end{align*}
The remaining cases where the potentials are equal over $[x_1, x_2]$ can be handled similarly.  Using analogous reasoning, it can also be shown that $H_t= 2 E$. This completes the proof.
\end{proof}
Now we show the second identity \eqref{wf2} of Definition \ref{intro-defn-weak formulation}. Again for a test function $\varphi$ with compact support in $(0, 1)\times (0, \infty)$, we have
\begin{align}\label{e3.30}
0=\int \int [H_x \varphi_{tx} (x, t)- H_t \varphi_{x x}(x, t)] dx dt&=\int \int [-q(x,t) \varphi_{t x}(x, t)- 2E \varphi_{x x}(x, t)] dx dt\nonumber\\
&=\int \int u (x,t) \varphi_t (x, t)dm dt +\int \int u^2(x,t)\varphi_x(x, t)dm dt\,.
\end{align}
The identity \eqref{e3.30} proves that $(\rho, u)$ satisfies the second equation of the system in the sense of Definition \ref{intro-defn-weak formulation}.

\subsection{Boundary condition and entropy admissibility}\label{sec:4}
In this subsection, we verify that the constructed solution satisfies the initial and boundary conditions. Furthermore, we show the entropy admissibility of the solution up to the boundary.
\subsubsection{Boundary condition} Our main focus is to verify the boundary conditions. The solution satisfies the boundary conditions in the following sense: when the net flux enters the domain from the boundary, the conditions are satisfied in its classical sense. In contrast, when the net flux exits the domain, the mass accumulates at the boundary, ensuring overall mass conservation.
\begin{defn}\label{def:sol}
Let $m$ be as in Definition~\ref{d3}. Then we define
\begin{equation*}
\rho(x,t)=\left\{\begin{aligned}
&\partial_x m&&\text{for } 0<x<1,\text{ and } t>0\\
&\lim_{x\searrow 0}\rho(x,t) &&\text{for }x=0\text{ and }\min\left\{F(0,t), G_{b_r}(0, t)\right\}>G_{b_l}(0,t),\\
&\delta\cdot\lim_{x\searrow 0}\left(m(x,t)-m(0,t)\right)&&\text{for }x=0 \text{
and }\min\left\{F(0, t), G_{b_r}(0,t)\right\}\leq G_{b_l}(0,t),\\
& \lim_{x\nearrow 1}\rho(x,t) &&\text{for }x=1\text{ and }\min\left\{F(1,t), G_{b_l}(1, t)\right\}>G_{b_r}(1,t),\\
&\delta\cdot\lim_{x\nearrow 1}\left(m(1,t)-m(x,t)\right)&&\text{for }x=1\text{ and }\min\left\{F(1,t), G_{b_l}(1, t)\right\}\leq G_{b_r}(1,t).
\end{aligned}\right.
\end{equation*}
Here $\delta$ is the Dirac measure, $\partial_x$ is the distributional derivative and $\rho$ is interpreted as a measure.
\end{defn}
\begin{lem}\label{thm: bd condition}
 The pair $(\rho,u)$ as defined above solves equation \eqref{e1.1} in $(0,1)\times (0, \infty)$.\\
The initial conditions are satisfied in the sense that for almost all $x$ we have $\lim_{t\searrow 0} u(x,t)= u_0(x)$ and  $\rho=\partial_{x}m$ with $\lim_{t\searrow 0}m(x,t) = \int_0^x \rho_0(y) dy$.\\
The boundary conditions are satisfied in regions where $\min\left\{F(0,t), G_{b_r}(0,t)\right\} > G_{b_l}(0,t)$ and\\ $\min\left\{F(1,t), G_{b_l}(1,t)\right\} > G_{b_r}(1,t)$ in the sense that for almost all $t$ we have $\lim_{x\searrow 0} u(x,t) = u_{b_l}(t)$ and $\lim_{x\nearrow 1} u(x,t) = u_{b_r}(t).$\\
In addition, if $u_{b_l}, u_{b_r}$ are in $C^1_{loc}$ and $\rho_{b_l}, \rho_{b_r}$ are locally Lipschitz continuous, then $\lim_{x\searrow 0} \rho(x,t)u(x,t) = \rho_{b_l}(t)u_{b_l}(t)$ and $\lim_{x\nearrow 1} \rho(x,t)u(x,t) = \rho_{b_r}(t)u_{b_r}(t).$  
\end{lem}
\begin{proof}
The validity of the initial condition and the left boundary condition is proved in \cite{WHD97} and \cite{NOSS22}. We will only check the right boundary condition. Notice that, from our assumption $\min\left\{F(1, t), G_{b_l}(1, t)\right\}>G_{b_r}(1, t)$ and the continuity of $F, G_{b_l}$ and $G_{b_r},$ we have $\min\left\{F(x, t), G_{b_l}(x, t)\right\}>G_{b_r}(x, t)$ holds for $x$ close to the boundary $\{x=1\}.$ Let $t_0$ be any Lebesgue point of $\rho_{b_r}$ and $u_{b_r}.$ First, we show that the boundary condition holds for $u,$ i.e. $\lim_{x \to 1}u(x, t_0)=u_{b_r}(t_0).$ Let $(x_n, t_0) \to (1, t_0).$ First, we assume that $\xi_{*}(x_n, t_0)=\xi^{*}(x_n, t_0)=\xi(x_n, t_0).$ In this case, the minimum is unique and for any $h \neq 0,$ we have
\begin{align*}
    G_{b_r}\left(\xi(x_n, t_0), x_n, t_0\right)+F(1, x_n, t_0)< G_{b_r}\left(\xi(x_n, t_0)+h(x_n-1), x_n, t_0\right)+F(1, x_n, t_0).
\end{align*}
Simplification of the above inequality gives,
\begin{align*}
    &\int_{\xi(x_n, t_0)}^{\xi(x_n, t_0)+h(x_n-1)}(x_n-1)u_{b_r}(\eta)\rho_{b_r}(\eta)\,d\eta\\
    &> \int_{\xi(x_n, t_0)}^{\xi(x_n, t_0)+h(x_n-1)} (t_0-\eta)u^2_{b_r}(\eta)\rho_{b_r}(\eta)\, d\eta\\
    &\geq\left(t_0-\xi(x_n, t_0)-h(x_n-1)\right)\int_{\xi(x_n, t_0)}^{\xi(x_n, t_0)+h(x_n-1)} u^2_{b_r}(\eta)\rho_{b_r}(\eta)\,d\eta.
\end{align*}
From the Definition \ref{d2}, we have
\begin{align*}
    \left(\frac{1}{u(x_n, t_0)}-h\right)\frac{\int_{\xi(x_n, t_0)}^{\xi(x_n, t_0)+h(x_n-1)}u^2_{b_r}(\eta)\rho_{b_r}(\eta)\,d\eta}{\int_{\xi(x_n, t_0)}^{\xi(x_n, t_0)+h(x_n-1)}u_{b_r}(\eta)\rho_{b_r}(\eta)\,d\eta} \leq 1,
\end{align*}
and passing to the limit $n \to \infty,$ we get
\begin{align}\label{EQQ4.2}
    \lim \sup_{n \to \infty}\left(\frac{1}{u(x_n, t_0)}-h\right)u_{b_r}(t_0)\leq 1.
\end{align}
Similarly, considering the inequality
\begin{align*}
G_{b_r}\left(\xi(x_n, t_0), x_n, t_0\right)< G_{b_r}\left(\xi(x_n, t_0)-h(x_n-1), x_n, t_0\right),  
\end{align*}
we obtain
\begin{align}\label{EQQ4.3}
    \lim \inf_{n \to \infty}\left(\frac{1}{u(x_n, t_0)}+h\right)u_{b_r}(t_0)\geq 1.
\end{align}
Since $h$ is arbitrary, combining the inequalities \eqref{EQQ4.2}-\eqref{EQQ4.3}, we get
\begin{align*}
    \lim_{n \to \infty}\frac{u_{b_r}(t_0)}{u(x_n, t_0)}=1,
\end{align*}
and this completes the proof of $\lim_{n \to \infty}u(x_n, t_0)=u_{b_r}(t_0).$

Next, we consider the case $\xi_{*}(x_n, t_0)< \xi^{*}(x_n, t_0).$ By Definition \ref{d2}, we have
\begin{align*}
    u(x_n, t_0)=\frac{\int_{\xi_*(x_n, t_0)}^{\xi^*(x_n, t_0)}\rho_{b_r}u^2_{b_r}\, d\eta}{\int_{\xi_*(x_n, t_0)}^{\xi^*(x_n, t_0)}\rho_{b_r}u_{b_r}\, d\eta}.
\end{align*}
Since $t_0$ is the Lebesgue point of $\rho_{b_r}u^2_{b_r}$ and $\rho_{b_r}u_{b_r},$ we find $\lim_{n\to \infty}u(x_n, t_0)=u_{b_r}(t_0).$

It remains to show that the right boundary condition for $\rho$ holds. In this case, if the boundary potential $$G_{b_r}(\xi,x,t)=\int_{0}^{\xi} [x-1-u_b(\eta)(t-\eta)]\rho_b(\eta)u_b(\eta)d\eta+ F(1, x, t)$$ attains the minimum in $(0,\infty)$ at a point $\bar{\xi}$, then $\frac{\partial G_{b_r}}{\partial \tau}\Big|_{\xi=\bar{\xi}}=0.$ That is,
\[
\frac{\partial G_{b_r}}{\partial \tau}\Big|_{\xi=\bar{\xi}}=\left(x-1-u_{b_r}(\bar{\xi})(t-\bar{\xi})\right)\rho_{b_r}(\bar{\xi})u_{b_r}(\bar{\xi})=0\,.
\]
Since $\xi_{*}(x,t), \xi^{*}(x,t)\nearrow t$  as $x \nearrow 1$,  for $x$ close enough to $1$, the minimizing point lies in $(0, \infty)$.
Consider the function $g:\mathbb{R^+}\times[0,\infty)\to\mathbb{R}$ where $g$ is defined as
\[
g(x,\xi)=x-1-u_{b_r}(\xi)(t-\xi)\,.
\]
Then $\frac{\partial g}{\partial \xi}=-(t-\xi) u_{b_r}^{\prime}(\xi)+u_{b_r}(\xi)$, $g(1,t)=0$ and $\frac{\partial g}{\partial \xi}(1,t)<0.$ Thus by the implicit function theorem there exists a neighborhood of $(1,t)$ where $g(x,\xi)=0$ has a unique solution. That gives the unique minimizer of the boundary potential $G_{b_r}(\xi,x,t)$ and thus $\xi_*(x,t)$ is a continuously differentiable function of $x$ in some neighborhood of $x=1$. Now from the relation $u(x,t)=\frac{x-1}{t-\xi_*(x,t)}$, we see that $\lim_{x\nearrow 1}\frac{\partial}{\partial x}\xi_*(x,t)=-\frac{1}{u_{b_r}(t)}.$
Consequently,
\begin{equation*}
\lim_{x\nearrow 1}\rho(x,t)=\lim_{x\nearrow 1}\frac{\partial}{\partial x}m(x,t)
=-\lim_{x\nearrow 1}\rho_{b_r}(\xi_*(x,t))u_{b_r}(\xi_*(x,t))\frac{\partial}{\partial x}\xi_*(x,t)
=\rho_{b_r}(t)\,.
\end{equation*}
This completes the verification of the initial and boundary condition for $(\rho, u)$.
\end{proof}
\subsubsection{Entropy condition} We show that the solution satisfies the Oleinik entropy condition in $\Omega.$
\begin{lem}\label{thm: entropy}
    Let $0<x<1,$ and $t>0.$ If $x$ is a point of discontinuity of $u(\cdot, t),$ then we have
    \begin{align*}
        u(x-, t)>u(x, t)> u(x+, t).
    \end{align*}
    Moreover, for almost all $t>0, $ we have the following:
    \begin{enumerate}
        \item [(i)] $u(\cdot, t)$ is either right continuous at $x=0$ or $u(0, t)>u(0+, t),$
        \item[(ii)] $u(\cdot, t)$ is either left continuous at $x=1$ or $u(1-, t)>u(1, t).$ 
    \end{enumerate}
\end{lem}
\begin{proof}
    We consider the points $(x, t),$ where $G_{b_r}(x, t)< \min\left\{F(x, t), G_{b_l}(x, t)\right\}$ and we have $\xi_*(x, t)<\xi^*(x, t).$ Note that, using the semicontinuity of $\xi(x, t),$ we also have $u(x-, t)=\frac{x-1}{t-\xi_*(x, t)}$ and $u(x+, t)=\frac{x-1}{t-\xi^*(x,t)}.$ From
\[
G_{b_r}(\xi_{*}(x,t), x, t)=G_{b_r}(\xi^{*}(x,t), x, t)\,,
\]
we know that
\[
\int_{\xi_{*}(x,t)}^{\xi^{*}(x,t)} [x-1-u_{b_r}(\eta)(t-\eta)]\rho_{b_r}(\eta)u_{b_r}(\eta)d\eta=0\,.
\]
This implies
\[
 \frac{x-1}{t-\xi_{*}(x,t)}\int_{\xi_{*}(x,t)}^{\xi^{*}(x,t)} \rho_{b_r}(\eta)u_{b_r}(\eta) d\eta=\int_{\xi_{*}(x,t)}^{\xi^{*}(x,t)} \frac{(t-\eta)} {t-\xi_{*}(x,t)} u^2_{b_r}(\eta) \rho_{b_r}(\eta) d\eta.
\]
Since $u_{b_r}<0,$ we can conclude $u(x-, t) > u(x,t)$. The other inequality also follows easily by dividing the above equality by $(t-\xi^{*}(x,t))$.

Next we consider the case $F(x, t)=G_{b_r}(x, t).$ If this happens in a region, then the solution is continuous and corresponds to a rarefaction wave from $(1, 0),$ see Remark \ref{rem2.9}.  Now we assume that the equality holds in the isolated points. We note that
\begin{align*}
    F(y^*(x, t), x, t)-G_{b_r}(\xi^*(x, t), x, t)< F(y^*(x, t), 1, \xi^*(x, t))-G_{b_r}(\xi^*(x, t), 1, \xi^*(x, t)).
\end{align*}
Simplifying the above inequality, we get
\begin{align*}
    \int_0^{y^*(x, t)}\left[(t-\xi^*(x,t))u_0(\eta)+(1-x)\right]\rho_0(\eta)\, d\eta&< \int_0^{\xi^*(x,t)}\left[x-1-u_{b_r}(t-\xi^*(x,t))\right]u_{b_r}(\eta)\rho_{b_r}(\eta)\,d\eta\\
    &+ \int_0^1 \left[(t-\xi^*(x,t))u_0(\eta)+(1-x)\right]\rho_0(\eta)\,d\eta.
\end{align*}
Dividing the above inequality by $t-\xi^*(x,t)$ in both side, we get
\[u(x, t)> \frac{x-1}{t-\xi^*(x,t)}:= u(x+, t).\]
Note that, to obtain the above inequality, we also used
\begin{align*}
    \int_0^{y^*(x, t)}\rho_0(\eta)\,d\eta+\int_0^{\xi^*(x,t)}u_{b_r}(\eta)\rho_{b_r}(\eta)\,d\eta-\int_0^1 \rho_0(\eta)\,d\eta < 0.
\end{align*}
To get the other way inequality, i.e., $\frac{x-y_*(x,t)}{t}=:u(x-, t)> u(x, t),$ we use
\begin{align*} 
G_{b_r}(\xi_{*} (x,t), x, t)- F(y_{*}(x,t),x, t)<G_{b_r}(\xi_{*}(x,t),y_{*} (x,t), 0)- F(y_{*}(x,t),y_{*}(x,t),0),
\end{align*}
and proceed similarly.

Now we consider the case $G_{b_l}(x,t)=G_{b_r}(x,t).$ By Remark \ref{rem2.9}, this can happen only at isolated points. In this case, we have $u(x+, t)=\frac{x-1}{t-\xi^*(x, t)}$ and $u(x-, t)=\frac{x}{t-\tau^*(x, t)}.$ Consider the inequality
\begin{align*}
    G_{b_l}(\tau^*(x, t), x, t)-G_{b_r}(\xi^*(x, t), x, t)< G_{b_l}(\tau^*(x, t), 1, \xi^*(x, t))-G_{b_r}(\xi^*(x, t), 1, \xi^*(x, t)).
\end{align*}
Simplifying the above inequality, we get
\begin{align*}
    \frac{x-1}{t-\xi^*(x,t)}\left[\int_0^{\tau^*(x, t)}u_{b_l}(\eta)\rho_{u_l}(\eta)\,d\eta-\int_0^{\xi^*(x, t)}u_{b_r}(\eta)\rho_{b_r}(\eta)\,d\eta+\int_0^1 \rho_0(\eta)\,d\eta\right]\\
    < \left[\int_0^{\tau^*(x, t)}u^2_{b_l}(\eta)\rho_{u_l}(\eta)\,d\eta-\int_0^{\xi^*(x, t)}u^2_{b_r}(\eta)\rho_{b_r}(\eta)\,d\eta+\int_0^1 u_0(\eta)\rho_0(\eta)\,d\eta\right]
\end{align*}
and this gives
\begin{align*}
    u(x+, t)< u(x, t).
\end{align*}
For the other inequality, we consider
\begin{align*}
    G_{b_r}(\xi^*(x, t), x, t)-G_{b_l}(\tau^*(x, t), x, t)< G_{b_r}(\xi^*(x, t), 0, \tau^*(x, t))-G_{b_l}(\tau^*, 0, \tau^*(x, t)).
\end{align*}
Similarly as above, simplifying the above, we obtain
\begin{align*}
    \frac{x}{t-\tau^*(x,t)}\left[\int_0^{\tau^*(x, t)}u_{b_l}(\eta)\rho_{b_l}(\eta)\, d\eta-\int_0^{\xi^*(x, t)}u_{b_r}(\eta)\rho_{b_r}(\eta)\,d\eta+\int_0^1 \rho_0(\eta)\,d\eta\right]\\
    > \left[\int_0^{\tau^*(x, t)}u^2_{b_l}(\eta)\rho_{b_l}(\eta)\,d\eta-\int_0^{\xi^*(x, t)}u^2_{b_r}(\eta)\rho_{b_r}(\eta)\,d\eta+\int_0^1 u_0(\eta)\rho_0(\eta)\,d\eta\right]
\end{align*}
which implies $u(x-, t)> u(x,t).$ This completes the proof of the interior case.

The first case $(i)$ is shown in \cite{NOSS22}. Now we show $(ii).$ If $\min \{F(1, t), G_{b_l}(1, t)\}>G_{b_r}(1, t),$ then the right continuity follows from Lemma \ref{thm: bd condition}. Next, consider the case $\min \{F(1, t), G_{b_l}(1, t)\}\leq G_{b_r}(1, t).$ Let $\min \{F(1, t), G_{b_l}(1, t)\}=G_{b_l}(1, t)< G_{b_r}(1, t).$  We note that, from the continuity of $G_{b_l}$ and $G_{b_r},$ it can be concluded that in a neighborhood of $(1, t),$ the inequality $G_{b_l}(x, t)< G_{b_r}(x, t)$ holds.

From the Definition \ref{d2}, in this neighborhood $u$ is given by
\begin{align*}
    u(x, t)=\begin{cases}
        \frac{x}{t-\tau^*(x, t)}\,\,\,\,\, &\text{if}\,\,\, \tau^*(x, t)=\tau_*(x, t),\\
        \frac{\int_{\tau_*(x, t)}^{\tau^*(x, t)}\rho_{b_l}u^2_{b_l}\, d\eta}{\int_{\tau_*(x, t)}^{\tau^*(x, t)}\rho_{b_l}u_{b_l}\, d\eta}\,\,\,\,\, &\text{if}\,\,\, \tau^*(x, t)<t_*(x, t).
    \end{cases}
\end{align*}
In the first case, $u(1-, t)=\frac{1}{t-\tau^*(1-, t)}>0.$ In the second case, we have
\begin{align*}
    \int_{\tau_*(x, t)}^{\tau^*(x, t)}\left(x-u_{b_l}(\eta)(t-\eta)\right)u_{b_l}(\eta)\rho_{b_l}(\eta)\, d\eta=0.
\end{align*}
From this, we get
\begin{align*}
    t u(x, t) \geq x+ \tau^*(x, t)\frac{\int_{\tau_*(x, t)}^{\tau^*(x, t)}u^2_{b_l}\rho_{b_l}\,d\eta}{\int_{\tau_*(x, t)}^{\tau^*(x, t)}u_{b_l}\rho_{b_l}\,d\eta} =x+ \tau^*(x, t)u(x, t).
\end{align*}
Passing to the limit $x\nearrow 1,$ we obtain $(t-\tau^*(1-, t))u(1-, t) \geq 1,$ which gives $u(1-, t)>0.$ Therefore, from the Definition \ref{d2}, we conclude $u(1-, t)> u(1, t).$ 

If $\min\{F(1, t), G_{b_l}(1, t)\}=F(1, t),$ again by the continuity of $F, G_{b_r},$ it follows that in a neighborhood of $(1, t),$ we have $F(x, t)<G_{b_r}(x, t).$ Now we can follow a similar argument as above, indeed we have,
\begin{align*}
    u(x, t)=\begin{cases}
        \frac{x-y_*(x, t)}{t}\,\,\,\,\, &\text{if}\,\,\, y^*(x, t)=y_*(x, t),\\
        \frac{\int_{y_*(x, t)}^{y^*(x, t)}\rho_0u_0\, d\eta}{\int_{y_*(x, t)}^{y^*(x, t)}\rho_{0}\, d\eta}\,\,\,\,\, &\text{if}\,\,\, y^*(x, t)<y_*(x, t).
    \end{cases}
\end{align*}
In the first case, passing to the limit $x\nearrow 1,$ we get, $u(1-, t)=\frac{1-y_*(1, t)}{t}$ which is strictly positive except $y_*(1, t)=1.$ Suppose, $y_*(1, t)=1,$ then we have
\begin{align*}
    F(1, t)=F(1, 1, t)\geq \int_0^{\xi}-(t-\eta)u^2_{b_r}(\eta)\rho_{b_r}(\eta)\,d\eta+F(1,1,t)=G_{b_r}(1, t),
\end{align*}
which is a contradiction to our assumption $F(1, t)<G_{b_r}(1, t).$ In the second case, we note
\begin{align*}
    \int_{y_*(x, t)}^{y^*(x, t)}\left(t u_0(\eta)+\eta-x\right)\rho_0(\eta)\,d\eta=0.
\end{align*}
This implies 
\begin{align*}
    tu(x, t)-x=-\frac{\int_{y_*(x, t)}^{y^*(x, t)}\eta \rho_0(\eta)\, d\eta}{\int_{y_*(x, t)}^{y^*(x, t)}\rho_0(\eta)\, d\eta}\geq -y^*(x, t),
\end{align*}
and thus, passing to the limit $x\nearrow 1,$ we get $t u(1-, t)\geq 1-y^*(1-, t)=1-y_*(1, t)>0.$ This gives $u(1-, t)>0=u(1, t).$ Finally, we observe that the equality of potentials at the boundary can be handled similarly to the interior case. This completes the proof.
\end{proof}
\begin{proof}[Proof of Theorem \ref{TH1.3}]
 Combining \eqref{e3.22}, \eqref{e3.30}, Lemma \ref{thm: bd condition} and Lemma \ref{thm: entropy}, we conclude the proof.   
\end{proof}
\section{Explicit examples with Riemann data}\label{sec:5}
In this section, we provide some explicit examples when the initial and boundary data are of Riemann type. In our example, we show that the choice of initial and boundary data can lead to mass accumulation on the right boundary at $\left\{x=1\right\}$ and also on the left boundary $\{x=0\}.$
\subsection{Concentration of mass on the right boundary}
 We start our example with the following specific choice of initial and boundary data.
 \begin{align*}
 u_0=a, u_{b_l}=\tilde{b}, u_{b_r}=b, \rho_0=\rho_{b_l}=\rho_{b_r}=1, \,\,\, \text{with}\,\,\, a<b<0, \tilde{b}>0\,\,\, \text{and}\,\,\, \tilde{b}+a>0.
 \end{align*}
 Note that the restrictions on the data immediately imply $\tilde{b}+b>0.$
 From the definition of initial and boundary potentials, we get
 \begin{align*}
     &F(y, x, t)=\frac{1}{2}\left(y+(a t -x)\right)^2-\frac{1}{2}\left(a t -x\right)^2,\\
     &G_{b_l}(\tau, x, t)=\frac{1}{2}\left(\tilde{b}\tau+x-\tilde{b}t\right)^2-\frac{1}{2}\left(x-\tilde{b}t\right)^2, \\
     &G_{b_r}(\xi, x, t)=\frac{1}{2}\left(b \xi+(x-1-b t)\right)^2-\frac{1}{2}\left(x-1-b t\right)^2+(a t-x)+\frac{1}{2}.
 \end{align*}
 Moreover, we have
 \begin{align*}
 arg \min_{0\leq y \leq 1}F(y, x, t)=\begin{cases}
x-at\,\,\,\, &\text{if}\,\,\, 0\leq x-at\leq 1,\\
0\,\,\,\, &\text{if}\,\,\, x-at< 0,\\
1 \,\,\,\, &\text{if}\,\,\, 1< x-at,
\end{cases}
arg \min_{\tau \geq 0}G_{b_l}(\tau, x, t)=\begin{cases}
\frac{\tilde{b}t-x}{\tilde{b}}\,\,\,\, &\text{if}\,\,\, \tilde{b}t-x\geq 0,\\
0\,\,\,\, &\text{if}\,\,\, \tilde{b}t-x< 0,
\end{cases}
\end{align*}
and 
\begin{align*}
arg \min_{\xi \geq 0}G_{b_r}(\xi, x, t)=\begin{cases}
\frac{bt-x+1}{b}\,\,\,\, &\text{if}\,\,\, bt-x+1\leq 0,\\
0\,\,\,\, &\text{if}\,\,\, bt-x+1> 0.
\end{cases}    
\end{align*}
Therefore, we get
\begin{align*}
F(x, t)=\begin{cases}
-\frac{1}{2}\left(a t-x\right)^2\,\,\,\, &\text{if}\,\,\, 0\leq x-at\leq 1,\\
0\,\,\,\, &\text{if}\,\,\, x-at< 0,\\
\frac{1}{2}+a t-x \,\,\,\, &\text{if}\,\,\, 1< x-at,
\end{cases}
G_{b_l}(x, t)=\begin{cases}
-\frac{1}{2}\left(x-\tilde{b}t\right)^2\,\,\,\, &\text{if}\,\,\, \tilde{b}t-x\geq 0,\\
0\,\,\,\, &\text{if}\,\,\, \tilde{b}t-x< 0,
\end{cases}
\end{align*}
and 
\begin{align*}
G_{b_r}(x, t)=\begin{cases}
-\frac{1}{2}\left(x-1-bt\right)^2+(at-x)+\frac{1}{2}&\text{if}\,\,\, bt-x+1\leq 0,\\
at-x+\frac{1}{2}\,\,\,\, &\text{if}\,\,\, bt-x+1> 0.
\end{cases}
\end{align*}
Now we study the curves generated by the potentials.

\noindent\underline{Curve generated by the potentials $F$ and $G_{b_l}:$} We have the following $X_1(t)$ generated by the potentials $F$ and $G_{b_l}.$ Indeed, 
 \begin{align*}
     -\frac{1}{2}\left(at-x\right)^2=-\frac{1}{2}(x-\tilde{b}t)^2
 \end{align*}
 and this implies $X_1(t):=\frac{a+\tilde{b}}{2}t.$ Also note that $\dot{X}_1(t)>0.$ Moreover, mass on the curve $X_1(t)$ is $(\tilde{b}-a)t \delta.$ One can easily check the equality $q(x_1+, t)-q(x_1-, t)=\dot{X_1}(t)[m].$

\noindent\underline{Curve generated by the potentials $G_{b_l}$ and the rarefaction region $\{G_{b_r}=F\}:$} Similarly as above, we can find a curve $X_2(t)$ by
 \begin{align*}
     -\frac{1}{2}\left(x-\tilde{b}t\right)^2=\frac{1}{2}+at-x
 \end{align*}
 and this gives $X_2(t)=\tilde{b}t+1-\sqrt{2(\tilde{b}-a)t}.$ We have $\dot{X_2}(t)>0$ for $t\geq \tfrac{2}{\tilde{b}-a}$ and the mass along the curve is $\sqrt{2(\tilde{b}-a)t}\delta.$ Moreover, it can be checked $q(x_2+, t)-q(x_2-, t)=\dot{X_2}(t)[m].$ 

\noindent\underline{Curve generated by the potentials $G_{b_l}$ and $G_{b_r}:$} Again, we find a curve $X_3(t)$ from 
\begin{align*}
    -\frac{1}{2}\left(x-\tilde{b}t\right)^2=-\frac{1}{2}\left(x-1-bt\right)^2+\left(at-x\right)+\frac{1}{2}
\end{align*}
and thus $X_3(t):=\left(\frac{a-b}{\tilde{b}-b}\right)+\left(\frac{\tilde{b}+b}{2}\right)t.$ Mass along the curve can be calculated as $(\tilde{b}-b)t \delta.$ These  calculations leads us to the following Figure \ref{right bd fig}.
\begin{figure}[ht!]
\begin{minipage}{0.5\textwidth}
\begin{tikzpicture}[scale=1.55]
\draw[-] (0,0)--(4,0) node[anchor=north]{$x=1$};
\draw[-] (4,0)--(0,0) node[anchor=north]{$x=0$};
\draw[->] (0,0)--(0,6.3) node[anchor=east]{$t$};
\draw[->] (4,0)--(4,6.3) node[anchor=east]{$t$};
\draw  (2.5,0)node [anchor=north] {\parbox{3cm}{$\rho_0=1,$ $u_0=a$}};
\draw(-0.45,3.5)node[anchor=north, rotate=90]{$\rho_{b_l}=1$, $u_{b_l}=\tilde{b}$};
\draw(4,4.4)node[anchor=north, rotate=90]{$\rho_{b_r}=1$, $u_{b_r}=b$};
\draw (3.95,16/9)--(4.05,16/9);
\draw (4.6,16/9) node[anchor=east] {$\tfrac{2}{\tilde{b}-a}$};
\draw (3.95,16/5)--(4.05,16/5);
\draw (4.8,16/5) node[anchor=east] {$\tfrac{2(\tilde{b}-a)}{(\tilde{b}-b)^2}$};
\draw (3.95,16/2.8)--(4.05,16/2.8);
\draw (3.9,16/2.8) node[anchor=east] {$\tfrac{2(\tilde{b}-a)}{(\tilde{b}-b)(\tilde{b}+b)}$};
%
%
\draw (0,0) -- (16/30,16/9);
%
\draw (4,0)--(16/30,16/9);
\draw (4,0)--(16/10,16/5);
\draw (16/10, 16/5)--(4, 16/2.8);
\draw[scale=1, smooth] 
    (16/30,16/9) .. controls (16/22,16/7.1) .. (16/10,16/5);
\draw (16/6.5, 16/10)node[anchor=north, rotate=-40, teal]{$F=G_{b_r}=at-x+\frac{1}{2}$};
\draw (2.7, 16/4.5)node[anchor=north, rotate=0, teal]{\parbox{2cm}{$G_{b_r}=$\\$-\frac{1}{2}\left(x-1-bt\right)^2+\left(at-x+\frac{1}{2}\right)$}};
\draw (16/8, 16/18)node[anchor=north, rotate=0, teal]{\parbox{2cm}{$F=$\\$-\frac{1}{2}\left(at-x\right)^2$}};
\draw (16/12, 16/3)node[anchor=north, rotate=0, teal]{\parbox{2cm}{$G_{b_l}=$\\$-\frac{1}{2}\left(x-\tilde{b}t\right)^2$}};
\draw (16/25, 16/15)node[anchor=north, rotate=0, teal]{$X_1(t)$};
\draw (16/25, 16/5.5)node[anchor=north, rotate=0, teal]{$X_2(t)$};
\draw (16/6, 16/3.2)node[anchor=north, rotate=0, teal]{$X_3(t)$};
\end{tikzpicture}
\end{minipage}\hfill
\begin{minipage}{0.5\textwidth}
\begin{tikzpicture}[scale=1.55]
\draw[-] (0,0)--(4,0) node[anchor=north]{$x=1$};
\draw[-] (4,0)--(0,0) node[anchor=north]{$x=0$};
\draw[->] (0,0)--(0,6.3) node[anchor=east]{$t$};
\draw[->] (4,0)--(4,6.3) node[anchor=east]{$t$};
\draw  (2.5,0)node [anchor=north] {\parbox{3cm}{$\rho_0=1,$ $u_0=a$}};
\draw(-0.45,4.4)node[anchor=north, rotate=90]{$\rho_{b_l}=1$, $u_{b_l}=\tilde{b}$};
\draw(4,4.4)node[anchor=north, rotate=90]{$\rho_{b_r}=1$, $u_{b_r}=b$};
\draw (3.95,16/9)--(4.05,16/9);
\draw (4.6,16/9) node[anchor=east] {$\tfrac{2}{\tilde{b}-a}$};
\draw (3.95,16/5)--(4.05,16/5);
\draw (4.8,16/5) node[anchor=east] {$\tfrac{2(\tilde{b}-a)}{(\tilde{b}-b)^2}$};
\draw (3.95,16/2.8)--(4.05,16/2.8);
\draw (5.1,16/2.8) node[anchor=east] {$\tfrac{2(\tilde{b}-a)}{(\tilde{b}-b)(\tilde{b}+b)}$};
%
%
\draw [line width=0.5mm, teal] (0,0) -- (16/30,16/9);
\draw[->][line width=0.5mm, teal] (4, 16/2.8)--(4, 6.3);
%
\draw (4,0)--(16/30,16/9);
\draw (16/10, 16/5)--(4, 16/2.8);
\draw[scale=1, smooth, line width=0.5mm, teal] 
    (16/30,16/9) .. controls (16/22,16/7.1) .. (16/10,16/5);
\draw (4,0)--(16/10,16/5);
\draw [line width=0.5mm, teal](16/10, 16/5)--(4, 16/2.8);
\draw (16/10, 16/7)node[anchor=north, rotate=0, violet]{\parbox{2cm}{$m=1,$ $\rho=0,$ $u=\frac{x-1}{t}$}};
\draw (2.7, 16/4.5)node[anchor=north, rotate=0, violet]{\parbox{2cm}{$m=x-bt,$ $\rho=1,$ $u=b$}};
\draw (16/8, 16/21)node[anchor=north, rotate=0, violet]{\parbox{2cm}{$m=x-at,$ $\rho=1,$ $u=a$}};
\draw (16/12, 16/3)node[anchor=north, rotate=0, violet]{\parbox{2cm}{$m=x-\tilde{b}t,$ $\rho=1,$ $u=\tilde{b}$}};
\draw (16/60, 16/18)node[anchor=north, rotate=72, violet]{$\rho=(\tilde{b}-a)t \delta$};
\draw (16/40, 16/4.5)node[anchor=north, rotate=0, violet]{\parbox{0.9cm}{$\rho=\sqrt{2(\tilde{b}-a)t}\delta$}};
\draw (16/5.8, 16/3.2)node[anchor=north, rotate=45, violet]{$\rho=(\tilde{b}-b)t \delta$};
\end{tikzpicture}
\end{minipage}
\caption{Mass generating from origin, absorbing rarefaction wave and concentrating on the right boundary in a finite time.}\label{right bd fig}
\end{figure}
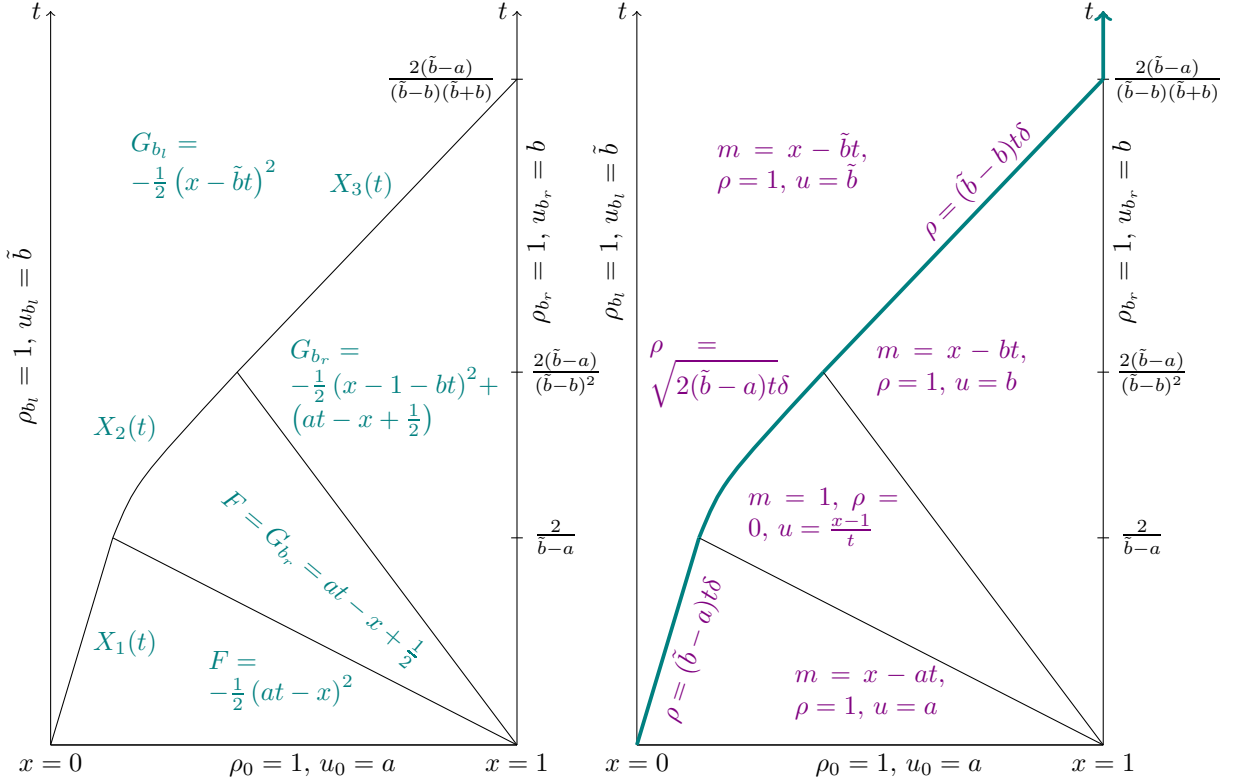
\subsection{Concentration of mass on the left boundary} In this example, we consider the initial and boundary data as
\begin{align*}
u_{b_l}=\tilde{b}, u_{b_r}=b, u_0=a, \rho_{b_r}=\rho_{b_l}=\rho_0=1,\,\,\, \text{with}\,\,\, 0<\tilde{b}<a, b<0, \,\,\, \text{and}\,\,\, 0<a<-b. 
\end{align*}
A similar calculation of potentials and curves leads to the following Figure \ref{left bd fig}.
\begin{figure}[ht!]
\begin{minipage}{0.5\textwidth}
\begin{tikzpicture}[scale=1.55]
\draw[-] (-4,0)--(0,0) node[anchor=north]{$x=1$};
\draw[-] (0,0)--(-4,0) node[anchor=north]{$x=0$};
\draw[->] (0,0)--(0,6.3) node[anchor=east]{$t$};
\draw[->] (-4,0)--(-4,6.3) node[anchor=east]{$t$};
\draw  (-1.5,0)node [anchor=north] {\parbox{3cm}{$\rho_0=1,$ $u_0=a$}};
\draw(0.1,4.5)node[anchor=north, rotate=90]{$\rho_{b_r}=1$, $u_{b_r}=b$};
\draw(-4.5,4.4)node[anchor=north, rotate=90]{$\rho_{b_l}=1$, $u_{b_l}=\tilde{b}$};
\draw (-3.95,16/9)--(-4.05,16/9);
\draw (-4.6,16/9) node[anchor=west] {$\tfrac{2}{a-b}$};
\draw (-3.95,16/5)--(-4.05,16/5);
\draw (-4.8,16/5) node[anchor=west] {$\tfrac{2(a-b)}{(\tilde{b}-b)^2}$};
\draw (-3.95,16/2.8)--(-4.05,16/2.8);
\draw (-3.9,16/2.8) node[anchor=west] {$\tfrac{2(b-a)}{(\tilde{b}-b)(\tilde{b}+b)}$};
\draw (0,0) -- (-16/30,16/9);
\draw(-16/10, 16/5)--(-4, 16/2.8);

\draw[scale=1, smooth] 
    (-16/30,16/9) .. controls (-16/22, 16/7) .. (-16/10,16/5);
\draw (-4, 0)--(-16/30,16/9);
\draw (-4, 0)--(-16/10,16/5);
\draw (-16/6.5, 16/10)node[anchor=north, rotate=40, teal]{$F=G_{b_l}=0$};
\draw (-2.7, 16/4.5)node[anchor=north, rotate=0, teal]{\parbox{2cm}{$G_{b_l}=$\\$-\frac{1}{2}\left(x-\tilde{b}t\right)^2$}};
\draw (-16/10, 16/18)node[anchor=north, rotate=0, teal]{\parbox{2cm}{$F=$\\$-\frac{1}{2}\left(at-x\right)^2$}};
\draw (-16/12, 16/3)node[anchor=north, rotate=0, teal]{\parbox{2cm}{$G_{b_r}=$\\$-\frac{1}{2}\left(x-1-bt\right)^2+(at-x)+\frac{1}{2}$}};
\draw (-16/25, 16/15)node[anchor=north, rotate=0, teal]{$X_1(t)$};
\draw (-16/35, 16/6)node[anchor=north, rotate=0, teal]{$X_2(t)$};
\draw (-16/6, 16/3.2)node[anchor=north, rotate=0, teal]{$X_3(t)$};
\end{tikzpicture}
\end{minipage}\hfill
\begin{minipage}{0.5\textwidth}
\begin{tikzpicture}[scale=1.55]
\draw[-] (-4,0)--(0,0) node[anchor=north]{$x=1$};
\draw[-] (0,0)--(-4,0) node[anchor=north]{$x=0$};
\draw[->] (0,0)--(0,6.3) node[anchor=east]{$t$};
\draw[->] (-4,0)--(-4,6.3) node[anchor=east]{$t$};
\draw  (-1.5,0)node [anchor=north] {\parbox{3cm}{$\rho_0=1,$ $u_0=a$}};
\draw(0.1,4.5)node[anchor=north, rotate=90]{$\rho_{b_r}=1$, $u_{b_r}=b$};
\draw(-4.5,4.4)node[anchor=north, rotate=90]{$\rho_{b_l}=1$, $u_{b_l}=\tilde{b}$};
\draw (-3.95,16/9)--(-4.05,16/9);
\draw (-4.6,16/9) node[anchor=west] {$\tfrac{2}{a-b}$};
\draw (-3.95,16/5)--(-4.05,16/5);
\draw (-4.8,16/5) node[anchor=west] {$\tfrac{2(a-b)}{(\tilde{b}-b)^2}$};
\draw (-3.95,16/2.8)--(-4.05,16/2.8);
\draw (-3.9,16/2.8) node[anchor=west] {$\tfrac{2(b-a)}{(\tilde{b}-b)(\tilde{b}+b)}$};
\draw [line width=0.5mm, teal] (0,0) -- (-16/30,16/9);
\draw [line width=0.5mm, teal] (-16/10, 16/5)--(-4, 16/2.8);
\draw[->] [line width=0.5mm, teal]  (-4, 16/2.8)--(-4, 6.3); 

\draw[scale=1, smooth, line width=0.5mm, teal] 
    (-16/30,16/9) .. controls (-16/22, 16/7) .. (-16/10,16/5);
\draw (-4, 0)--(-16/30,16/9);
\draw (-4, 0)--(-16/10,16/5);
\draw (-16/6.5, 16/10)node[anchor=north, rotate=40, violet]{$m=\rho=0,$ $u=\frac{x}{t}$};
\draw (-2.9, 16/5)node[anchor=north, rotate=0, violet]{\parbox{2cm}{$m=x-\tilde{b}t,$ $\rho=1,$ $u=\tilde{b}$}};
\draw (-16/10, 16/18)node[anchor=north, rotate=0, violet]{\parbox{2cm}{$m=x-at,$ $\rho=1,$ $u=a$}};
\draw (-16/12, 16/3)node[anchor=north, rotate=0, violet]{\parbox{2.5cm}{$m=x-bt,$$\rho=1,$$u=b$}};
\draw (-16/55, 16/15.5)node[anchor=north, rotate=-73, violet]{$\rho=(a-b)t \delta$};
\draw (-16/21, 16/4.8)node[anchor=north, rotate=0, violet]{\parbox{1.5cm}{$\rho=\sqrt{2(a-b)t}\delta$}};
\draw (-16/6, 16/3.2)node[anchor=north, rotate=-50, violet]{$\rho=(\tilde{b}-b)t \delta$};
\end{tikzpicture}
\end{minipage}
\caption{Mass generating from $(1, 0)$, absorbing rarefaction wave and concentrating on the left boundary in a finite time. Here $X_1(t)=\left(\frac{a+b}{2}\right)t+1,$ $X_2(t)=bt+\sqrt{2(a-b)t},$ $X_3(t)=\left(\frac{a-b}{\tilde{b}-b}\right)+\frac{\tilde{b}+b}{2}t.$}\label{left bd fig}
\end{figure}
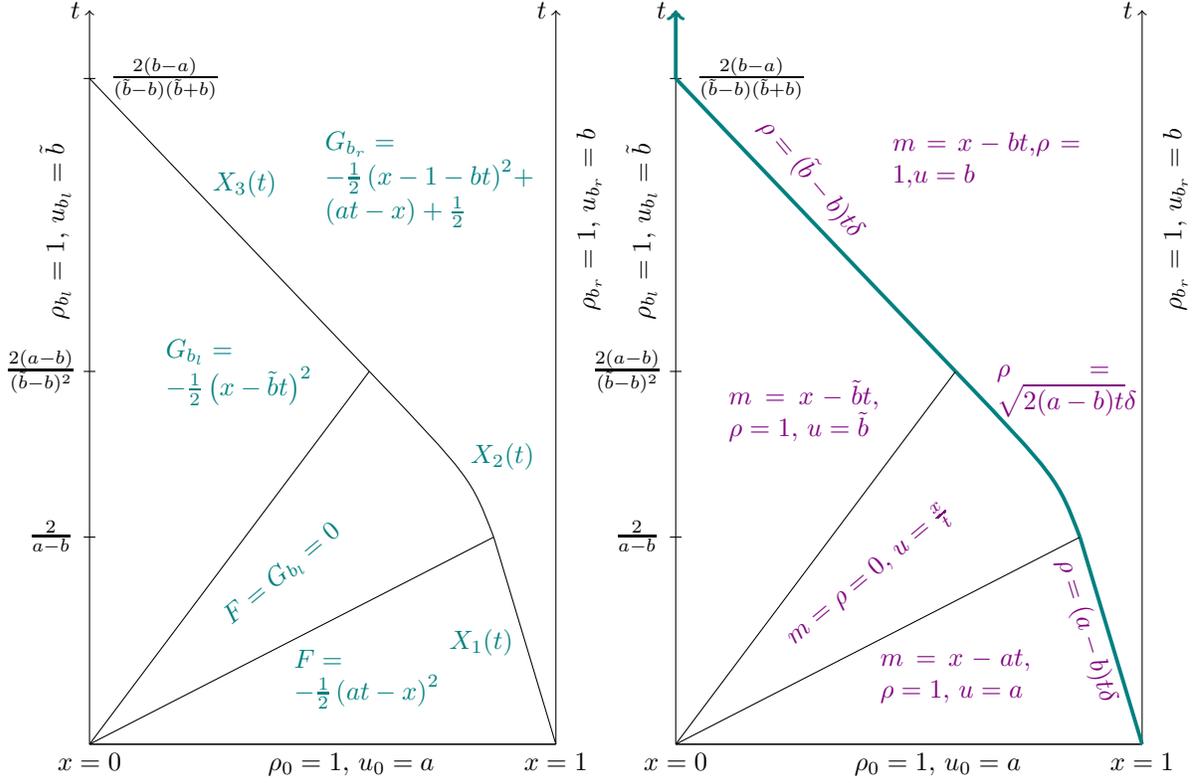

\subsection{Concentration of mass on curves inside the domain} In this case, we consider the following sets of boundary data:
\begin{align*}
\begin{cases}
\rho_{b_l}=\rho_{b_r}=1;\,\, u_{b_l}=3, u_{b_r}=-3\,\,\,\, &\text{for} \,\,\, t\geq \frac{1}{2},\\
\rho_{b_l}=\rho_{b_r}=1;\,\, u_{b_l}=2, u_{b_r}=-2 \,\,\,\, &\text{for} \,\,\, t< \frac{1}{2},
\end{cases}
\end{align*}
and initial datum: 
\begin{align*}
\begin{cases}
\rho_0=1; u_0=2\,\,\,\, &\text{for}\,\,\,x\leq \frac{1}{2},\\
\rho_0=1; u_0=-2, \,\,\,\, &\text{for}\,\,\,x> \frac{1}{2}.
\end{cases}
\end{align*}
This leads us to the Figure \ref{inside domain:fig}. Here, the initial data and the boundary data have jumps at $x=\frac{1}{2}$ and $t=\frac{1}{2},$ respectively. Thus, a Dirac delta with mass $4t \delta$ is generated from $x=\frac{1}{2}$ and Dirac deltas with mass $(t-\frac{1}{2}) \delta$ are generated from $t=\frac{1}{2}$ on the both of the boundaries. They meet at a point $(\frac{1}{2}, \frac{7}{10})$ and subsequently merge, propagating as a single Dirac delta with mass $(6t-1)\delta.$ 
\begin{figure}[ht!]
\begin{minipage}{0.5\textwidth}
\begin{tikzpicture}[scale=1.55]
\draw[-] (0,0)--(4,0) node[anchor=north]{$x=1$};
\draw[-] (4,0)--(0,0) node[anchor=north]{$x=0$};
\draw[->] (0,0)--(0,6.3) node[anchor=east]{$t$};
\draw[->] (4,0)--(4,6.3) node[anchor=east]{$t$};
\draw  (1.3,0)node [anchor=north] {\parbox{2cm}{$\rho_0=1$\\$u_0=2$}};
\draw  (3.0,0)node[anchor=north] {\parbox{2cm}{$\rho_0=1$\\$u_0=-2$}};
\draw (2,0) node[anchor=north] {$\frac{1}{2}$};
\draw (2,-.05)--(2,0.05);
\draw(-0.35,4.5)node[anchor=north, rotate=90]{$\rho_{b_l}=1$, $u_{b_l}=3$};
\draw(-0.35,1.3)node[anchor=north, rotate=90]{$\rho_{b_l}=1$, $u_{b_l}=2$};
\draw(4.05, 5.3)node[anchor=north, rotate=90]{$\rho_{b_r}=1$, $u_{b_r}=-3$};
\draw(4.05,1.3)node[anchor=north, rotate=90]{$\rho_{b_r}=1$, $u_{b_r}=-2$};
\draw (3.95,16/4)--(4.05,16/4);
\draw (4.5,16/4) node[anchor=east] {$\tfrac{7}{10}$};
\draw (-0.05,16/6)--(0.05,16/6);
\draw (0,16/6) node[anchor=east] {$\tfrac{1}{2}$};
\draw (3.95,16/6)--(4.05,16/6);
\draw (4.3,16/6) node[anchor=east] {$\tfrac{1}{2}$};
\draw[->] (2,0)--(2, 16/3);
\draw (0, 16/6)--(2, 16/4);
\draw (4, 16/6)--(2, 16/4);
\draw  (2,0.8)node[anchor=north, rotate=90]{\parbox{2cm}{$X_1(t)=1/2$}};
\draw  (.7, 16/4.3)node[anchor=north, rotate=35]{\parbox{2.5cm}{$X_2(t)=\frac{5}{2}(t-\frac{1}{2})$}};
\draw  (3.2, 3.75)node[anchor=north, rotate=-33]{\parbox{4cm}{$X_3(t)=-\frac{5}{2}(t-\frac{1}{2})+1$}};
%
%
%
\draw (1, 2.45)node[anchor=north, teal]{\parbox{2cm}{$F=G_{b_l}=$\\$-\frac{1}{2}\left(x-2t\right)^2$}};
\draw (3, 2.45)node[anchor=north, teal]{\parbox{2.5cm}{$F=G_{b_r}=2t$\\$-\frac{1}{2}\left(x+2t\right)^2$}};
\draw (3, 5.5)node[anchor=north, teal]{\parbox{2.5cm}{$G_{b_r}=-\frac{1}{2}(x-1-3t)^2$\\$-\frac{1}{2}x+\frac{5}{2}t-\frac{5}{8}$}};
\draw (1.05, 5)node[anchor=north, teal]{\parbox{3cm}{$G_{b_l}=-\frac{1}{2}(x-5t)$\\$-\frac{1}{2}\left(x-3t\right)^2-\frac{5}{8}$}};
\end{tikzpicture}
\end{minipage}\hfill
\begin{minipage}{0.5\textwidth}
\begin{tikzpicture}[scale=1.55]
\draw[-] (0,0)--(4,0) node[anchor=north]{$x=1$};
\draw[-] (4,0)--(0,0) node[anchor=north]{$x=0$};
\draw[->] (0,0)--(0,6.3) node[anchor=east]{$t$};
\draw[->] (4,0)--(4,6.3) node[anchor=east]{$t$};
\draw  (1.3,0)node [anchor=north] {\parbox{2cm}{$\rho_0=1$\\$u_0=2$}};
\draw  (3.0,0)node[anchor=north] {\parbox{2cm}{$\rho_0=1$\\$u_0=-2$}};
\draw (2,0) node[anchor=north] {$\frac{1}{2}$};
\draw (2,-.05)--(2,0.05);
\draw(-0.35,4.5)node[anchor=north, rotate=90]{$\rho_{b_l}=1$, $u_{b_l}=3$};
\draw(-0.35,1.3)node[anchor=north, rotate=90]{$\rho_{b_l}=1$, $u_{b_l}=2$};
\draw(4.05, 5.3)node[anchor=north, rotate=90]{$\rho_{b_r}=1$, $u_{b_r}=-3$};
\draw(4.05,1.3)node[anchor=north, rotate=90]{$\rho_{b_r}=1$, $u_{b_r}=-2$};
\draw (3.95,16/4)--(4.05,16/4);
\draw (4.5,16/4) node[anchor=east] {$\tfrac{7}{10}$};
\draw (-0.05,16/6)--(0.05,16/6);
\draw (0,16/6) node[anchor=east] {$\tfrac{1}{2}$};
\draw (3.95,16/6)--(4.05,16/6);
\draw (4.3,16/6) node[anchor=east] {$\tfrac{1}{2}$};
\draw[line width=0.5mm, teal][->] (2,0)--(2, 16/3);
\draw [line width=0.5mm, teal](0, 16/6)--(2, 16/4);
\draw [line width=0.5mm, teal](4, 16/6)--(2, 16/4);
\draw  (2,1.3)node[anchor=north, rotate=90, violet]{\parbox{2cm}{$\rho=4t \delta$}};
%
%
%

\draw (1, 2.45)node[anchor=north, violet]{\parbox{2cm}{$m=x-2t,$ $\rho=1,$ $u=2$}};
\draw (3, 2.45)node[anchor=north, violet]{\parbox{2.5cm}{$m=x+2t,$ $\rho=1,$ $u=-2$}};
\draw (3, 5)node[anchor=north, violet]{\parbox{2.5cm}{$m=\frac{1}{2}+(x-1+3t),$ $\rho=1,$ $u=-3$}};
\draw (1.05, 5)node[anchor=north, violet]{\parbox{3cm}{$m=\frac{1}{2}+(x-3t),$ $\rho=1,$ $u=3$}};
\draw  (.7, 16/4.3)node[anchor=north, rotate=35, violet]{\parbox{2.5cm}{$\rho=(t-\frac{1}{2})\delta$}};
\draw  (3.5, 3)node[anchor=north, rotate=-33, violet]{\parbox{4cm}{$\rho=(t-\frac{1}{2})\delta$}};
\draw  (2.5, 5.7)node[anchor=north, rotate=0, violet]{\parbox{4cm}{$\rho=(6t-1)\delta$}};
\draw	(2, 16/4) circle[radius=1.2pt];
\fill (2,16/4) circle[radius=1.2pt];
\end{tikzpicture}
\end{minipage}
\caption{Mass generating from the initial manifold, and two boundary manifolds, merging at $(\frac{1}{2}, \frac{7}{10})$ inside the domain.}\label{inside domain:fig}
\end{figure}
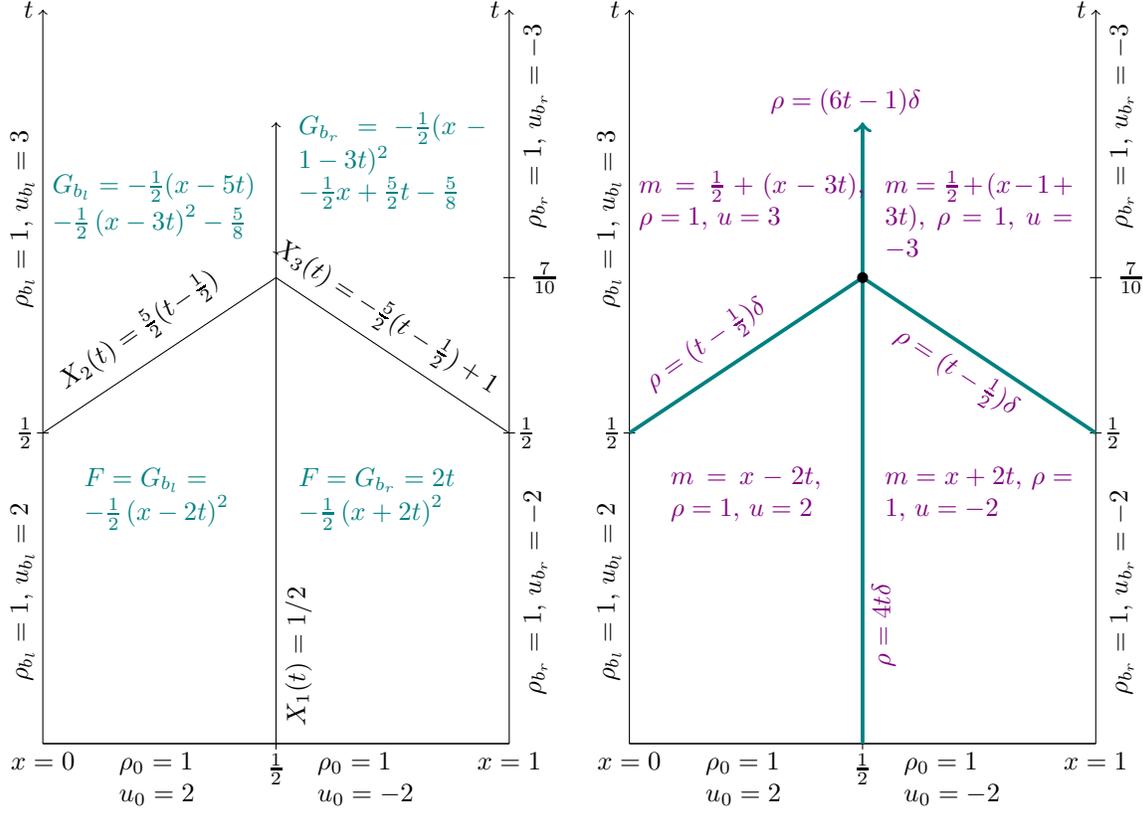
\subsection{Two rarefactions absorbed by one shock} We construct an example illustrating how rarefactions originating from the points $(0, 0)$ and $(1, 0)$ are absorbed by a Dirac delta emerging from the initial manifold. The initial and boundary data are given by:
\begin{align*}
    u_0=\begin{cases}
        2\,\,\, &\text{for}\,\,\, x\leq \frac{1}{2},\\
        -2\,\,\, &\text{for}\,\,\, x> \frac{1}{2}.
    \end{cases}
\,\,\,\,\,\,\,\,\,\,\, u_{b_l}=1, \,\, u_{b_r}=-1, \,\,\, \rho_{b_l}=\rho_{b_r}=\rho_0=1.  
\end{align*}
In this case, we have two rarefaction regions bounded by the lines $x=t; x=2t$ and $x=-t+1; x=-2t+1.$ A Dirac mass $4t\delta$ is generated from $X_1=\frac{1}{2}$ and propagates along the straight line $X_1=\frac{1}{2}$. For $t>\frac{1}{2},$ it evolves into $2t \delta,$ effectively absorbing the rarefactions.

Notice that, here the initial and boundary data are considered such a way that the rarefaction regions starts from the corner points of the domain $\Omega.$ A slight modification in the data can yield similar examples where the rarefactions emerge from the boundary manifold instead. Indeed, if we fix a time level $t=\varepsilon_0>0 $ and consider the boundary data as:
\begin{align*}
\begin{cases}
\rho_{b_l}=\rho_{b_r}=1;\,\, u_{b_l}=1, u_{b_r}=-1\,\,\,\, &\text{for} \,\,\, t\geq \varepsilon_0,\\
\rho_{b_l}=\rho_{b_r}=1;\,\, u_{b_l}=2, u_{b_r}=-2 \,\,\,\, &\text{for} \,\,\, t< \varepsilon_0,
\end{cases}
\end{align*}
then rarefactions will be initiated from the points $(0, \varepsilon_0)$ and $(1, \varepsilon_0)$ which will ultimately be absorbed by the delta shock propagating from $x=\frac{1}{2}.$
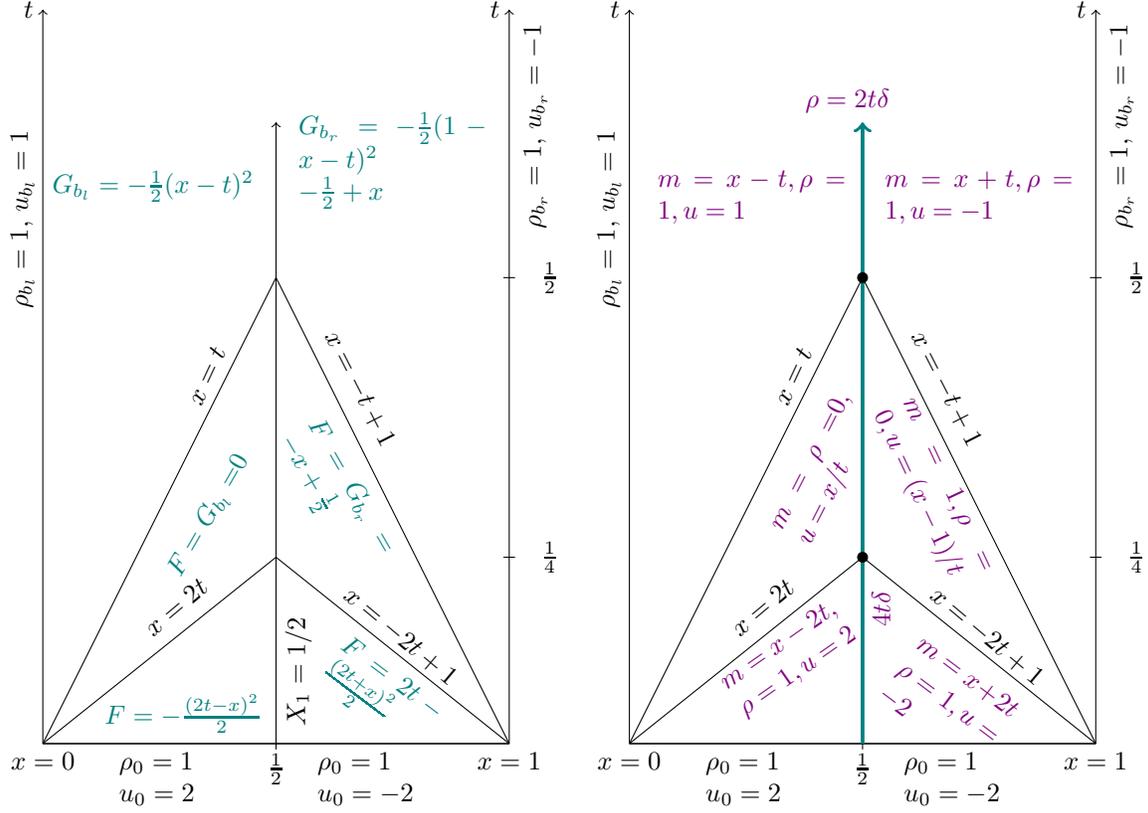
\begin{figure}[ht!]
\begin{minipage}{0.5\textwidth}
\begin{tikzpicture}[scale=1.55]
\draw[-] (0,0)--(4,0) node[anchor=north]{$x=1$};
\draw[-] (4,0)--(0,0) node[anchor=north]{$x=0$};
\draw[->] (0,0)--(0,6.3) node[anchor=east]{$t$};
\draw[->] (4,0)--(4,6.3) node[anchor=east]{$t$};
\draw  (1.3,0)node [anchor=north] {\parbox{2cm}{$\rho_0=1$\\$u_0=2$}};
\draw  (3.0,0)node[anchor=north] {\parbox{2cm}{$\rho_0=1$\\$u_0=-2$}};
\draw (2,0) node[anchor=north] {$\frac{1}{2}$};
\draw (2,-.05)--(2,0.05);
\draw(-0.35,4.5)node[anchor=north, rotate=90]{$\rho_{b_l}=1$, $u_{b_l}=1$};
\draw(4.05, 5.3)node[anchor=north, rotate=90]{$\rho_{b_r}=1$, $u_{b_r}=-1$};
\draw (3.95,16/4)--(4.05,16/4);
\draw (4.5,16/4) node[anchor=east] {$\tfrac{1}{2}$};
\draw (3.95,16/10)--(4.05,16/10);
\draw (4.5,16/10) node[anchor=east] {$\tfrac{1}{4}$};
\draw[->] (2,0)--(2, 16/3);
\draw (0, 0)--(2, 16/10);
\draw (0, 0)--(2, 16/4);
\draw (4, 0)--(2, 16/10);
\draw (4, 0)--(2, 16/4);
\draw  (2,0.8)node[anchor=north, rotate=90]{\parbox{2cm}{$X_1=1/2$}};
\draw  (1.5, 16/4.3)node[anchor=north, rotate=65]{\parbox{2.5cm}{$x=t$}};
\draw  (3.2, 2.45)node[anchor=north, rotate=-63]{\parbox{4cm}{$x=-t+1$}};
\draw  (1.3, 1.5)node[anchor=north, rotate=40]{\parbox{2cm}{$x=2t$}};
\draw  (3.7, .6)node[anchor=north, rotate=-40]{\parbox{4cm}{$x=-2t+1$}};
%
%
%
\draw (1.3, 2.1)node[anchor=north, rotate=60, teal]{\parbox{2cm}{$F=G_{b_l}=$$0$}};
\draw (2.8, 2.30)node[anchor=north, rotate=-60, teal]{\parbox{2cm}{$F=G_{b_r}=-x+\frac{1}{2}$}};
\draw (3, 5.5)node[anchor=north, teal]{\parbox{2.5cm}{$G_{b_r}=-\frac{1}{2}(1-x-t)^2$\\$-\frac{1}{2}+x$}};
\draw (1.05, 5)node[anchor=north, teal]{\parbox{3cm}{$G_{b_l}=-\frac{1}{2}(x-t)^2$}};
\draw (1.5, .5)node[anchor=north, teal]{\parbox{3cm}{$F=-\frac{(2t-x)^2}{2}$}};
\draw (3.1, .7)node[anchor=north, rotate=-37,  teal]{\parbox{1.6cm}{$F=2t-\frac{(2t+x)^2}{2}$}};
\end{tikzpicture}
\end{minipage}\hfill
\begin{minipage}{0.5\textwidth}
\begin{tikzpicture}[scale=1.55]
\draw[-] (0,0)--(4,0) node[anchor=north]{$x=1$};
\draw[-] (4,0)--(0,0) node[anchor=north]{$x=0$};
\draw[->] (0,0)--(0,6.3) node[anchor=east]{$t$};
\draw[->] (4,0)--(4,6.3) node[anchor=east]{$t$};
\draw  (1.3,0)node [anchor=north] {\parbox{2cm}{$\rho_0=1$\\$u_0=2$}};
\draw  (3.0,0)node[anchor=north] {\parbox{2cm}{$\rho_0=1$\\$u_0=-2$}};
\draw (2,0) node[anchor=north] {$\frac{1}{2}$};
\draw (2,-.05)--(2,0.05);
\draw(-0.35,4.5)node[anchor=north, rotate=90]{$\rho_{b_l}=1$, $u_{b_l}=1$};
\draw(4.05, 5.3)node[anchor=north, rotate=90]{$\rho_{b_r}=1$, $u_{b_r}=-1$};
\draw (3.95,16/4)--(4.05,16/4);
\draw (4.5,16/4) node[anchor=east] {$\tfrac{1}{2}$};
\draw (3.95,16/10)--(4.05,16/10);
\draw (4.5,16/10) node[anchor=east] {$\tfrac{1}{4}$};
\draw[line width=0.5mm, teal][->] (2,0)--(2, 16/3);
\draw (0, 0)--(2, 16/10);
\draw (0, 0)--(2, 16/4);
\draw (4, 0)--(2, 16/10);
\draw (4, 0)--(2, 16/4);
\draw  (2,1.65)node[anchor=north, rotate=90, violet]{\parbox{2cm}{$4t\delta$}};
\draw  (1.5, 16/4.3)node[anchor=north, rotate=65]{\parbox{2.5cm}{$x=t$}};
\draw  (3.2, 2.45)node[anchor=north, rotate=-63]{\parbox{4cm}{$x=-t+1$}};
\draw  (1.3, 1.5)node[anchor=north, rotate=40]{\parbox{2cm}{$x=2t$}};
\draw  (3.7, .6)node[anchor=north, rotate=-40]{\parbox{4cm}{$x=-2t+1$}};

\draw	(2, 16/4) circle[radius=1.2pt];
\fill (2,16/4) circle[radius=1.2pt];
\draw	(2, 16/10) circle[radius=1.2pt];
\fill (2,16/10) circle[radius=1.2pt];

%
%
%
\draw  (2.8, 5.7)node[anchor=north, rotate=0, violet]{\parbox{4cm}{$\rho=2t\delta$}};
\draw (1.4, 2.5)node[anchor=north, rotate=63, violet]{\parbox{2cm}{$m=\rho=$$0,$ $u=x/t$}};
\draw (2.89, 2.30)node[anchor=north, rotate=-65, violet]{\parbox{2.5cm}{$m=1, \rho=0, u=(x-1)/t$}};
\draw (3, 5)node[anchor=north, violet]{\parbox{2.5cm}{$m=x+t, \rho=1, u=-1$}};
\draw (1.05, 5)node[anchor=north, violet]{\parbox{2.5cm}{$m=x-t, \rho=1, u=1$}};
\draw (1.5, 1.2)node[anchor=north, rotate=35, violet]{\parbox{3cm}{$m=x-2t,$ \\$ \rho=1, u=2$}};
\draw (3, .7)node[anchor=north, rotate=-37,  violet]{\parbox{1.6cm}{$m=x+2t$\\$\rho=1, u=-2$}};
\end{tikzpicture}
\end{minipage}
\caption{Rarefaction generating from $(0, 0)$ and $(1, 0)$ absorbed by a shock generated from $x=\frac{1}{2}.$}\label{2 rarefaction:fig}
\end{figure}
\begin{rem}
(i)\, In Figure \ref{right bd fig}, the left boundary data $\tilde{b}$ is chosen to be sufficiently large and positive. This strong positive boundary data imposes a dominant influence, driving the entire mass toward the right boundary over time. On the other hand, in Figure \ref{left bd fig}, the right boundary data $b$ is set to be significantly negative. This negative boundary data effectively pushes all the mass toward the left boundary. These setups demonstrate how the choice of boundary conditions can control the direction of mass flow, ensuring that it is driven entirely toward the specified boundary.

\noindent (ii)\, In above the Figures \ref{right bd fig}-\ref{left bd fig}, we illustrate the values of $(\rho, u)$ within regions separated by rarefaction waves or shocks. Additionally, one can compute the value of $u$ at any point lying on the shock curve using the formula provided in Definition \ref{d2}.   For example, consider Figure \ref{right bd fig} and calculate the value of $u$ at a point $(x, t)$ located on the curve $X_3(t)$. Note that, at any point $(X_3(t), t),$ it holds that $G_{b_r}(x, t)=G_{b_l}(x, t).$ Thus, applying the formula given in Definition \ref{d2}, we have 
\begin{align*}
    u(x, t)&:=\frac{\int_0^{\tau^*(x, t)}\rho_{b_l} u^2_{b_l}-\int_0^{\xi^*(x, t)}\rho_{b_r}u^2_{b_r}+\int_0^1 \rho_0 u_0}{\int_0^{\tau^*(x, t)}\rho_{b_l} u_{b_l}-\int_0^{\xi^*(x, t)}\rho_{b_r}u_{b_r}+\int_0^1 \rho_0}.
\end{align*}
Substituting the explicit values of the initial and boundary datum, we obtain
\begin{align*}
u(x, t):=\frac{\int_0^{\frac{\tilde{b}t-x}{\tilde{b}}}\tilde{b}^2-\int_0^{\frac{bt-x+1}{b}}b^2+a}{\int_0^{\frac{\tilde{b}t-x}{\tilde{b}}}\tilde{b}-\int_0^{\frac{bt-x+1}{b}}b+1}.
\end{align*}
Simplifying further using the equation of the curve $X_3(t),$ we get:
\begin{align*}
    u(x, t)=\frac{\tilde{b}(\tilde{b}t-x)-b(bt-x+1)+a}{(\tilde{b}-b)t}=\frac{\tilde{b}+b}{2}.
\end{align*}
Note that this aligns with Lemma \ref{lem:curves}, where solutions are interpreted as the slopes of the curves and here $\frac{\tilde{b}+b}{2}$ is precisely the slope of the curve $X_3(t).$

\noindent(iii) We want to point out that there might be situations where $F(x, t)=G_{b_l}(x, t)=G_{b_r}(x, t).$ Figure \ref{inside domain:fig} illustrates such a case. For any points $(x,t)$ on the curve $X_1(t)$ with $t< \frac{7}{10},$ we observe that $F(x, t)=G_{b_l}(x, t)=G_{b_r}(x,t).$ The solution $u(x,t)$ along the curve $X_1(t), $ up to time $t=\frac{7}{10},$ can be computed using either \ref{it3} or \ref{it7} or \ref{it8} of Lemma \ref{lem:curves} and $u(x,t)$ turns out to be the same in every case. Indeed, using \ref{it3}, we get
\begin{align*}
    u(x, t):=\frac{\int_{0}^{\tau ^*}\rho_{b_l} u^2_{b_l}+\int_{0}^{y^*}\rho_0 u_{0}}{\int_{0}^{\tau^*}u_{b_l}\rho_{b_l}+\int_{0}^{y^*}\rho_0}=\frac{\int_0^{\frac{2t-x}{2}}4+\int_0^{\frac{1}{2}}2-\int_{\frac{1}{2}}^{2t+x}2}{\int_0^{\frac{2t-x}{2}}2+2t+x}=\frac{2-4x}{4t}.
\end{align*}
Using \ref{it7}, we get
\begin{align*}
   u(x, t):=\frac{\int_{0}^{\xi ^*}\rho_{b_r} u^2_{b_r}+\int_{0}^{y_*}\rho_0 u_{0}-\int_0^1 \rho_0 u_0}{\int_{0}^{\xi^*}u_{b_r}\rho_{b_r}+\int_{0}^{y^*}\rho_0-\int_0^1 \rho_0}=\frac{\int_0^{\frac{x-1+2t}{2}}4+\int_0^{x-2t}2}{\int_0^{\frac{x-1+2t}{2}}(-2)+\int_0^{x-2t}1-1}=\frac{2-4x}{4t}. 
\end{align*}
And lastly, using \ref{it8}, we get
\begin{align*}
    u(x, t):=\frac{\int_{0} ^ {\tau^{*}(x,t)}\rho_{b_l} u^2_{b_l}-\int_{0} ^ {\xi^{*}(x,t)}\rho_{b_r} u^2_{b_r}+\int_0^1 \rho_0u_0}{\int_{0} ^ {\tau^{*}(x,t)}\rho_{b_l} u_{b_l} -\int_{0} ^ {\xi^{*}(x,t)}\rho_{b_r} u_{b_r}+\int_0^1 \rho_0}=\frac{\int_0^{\frac{2t-x}{2}}4-\int_0^{\frac{x-1+2t}{2}}4}{\int_0^{\frac{2t-x}{2}}2+\int_0^{\frac{x-1+2t}{2}}2+1}=\frac{2-4x}{4t}. 
\end{align*}
Moreover, solutions along the curve $X_2(t), X_3(t)$ and along $X_1(t), t\geq \frac{7}{10}$ can be computed using \ref{it3}, \ref{it7} and \ref{it8} of Lemma \ref{lem:curves}. respectively.

\noindent (iv) Note that Figure \ref{2 rarefaction:fig} illustrates that for any $t \in [\frac{1}{4}, \frac{1}{2}],$ we may have situations where $F(x, t)=G_{b_l}(x, t)$ in $[l_1(t), r_1(t)]$ and $F(x, t)=G_{b_r}(x, t)$ in $[l_2(t), r_2(t)],$ with these two intervals being disjoint, i.e., $[l_1(t), r_1(t)]\cap [l_2(t), r_2(t)]=\emptyset.$ This scenario precisely corresponds to the final case discussed in Lemma \ref{lem: h function}. One more observation is that, at the time $t=\frac{1}{4},$ the mass $4t \delta$ becomes $\delta.$ Since there is no contribution of mass from the rarefaction regions, it then propagates along the straight line $X_1=\frac{1}{2}$ until the time $t=\frac{1}{2},$ at which point it evolves into $2t \delta.$ Moreover, for $t \in [\frac{1}{4}, \frac{1}{2}], $ $u$ can be calculated as identically zero along the line $X_1=\frac{1}{2}$ from the formula given in Definition \ref{d2}.
\end{rem}

\subsection*{Acknowledgments} The author is supported by postdoctoral research grants from the Alexander von Humboldt Foundation, Germany. He also thanks Lukas Neumann for helpful discussions and comments.

\subsection*{Conflict of interest} On behalf of all authors, the corresponding author states that there is no conflict of interest.

\subsection*{Data availability} This manuscript has no associated data.

\bibliographystyle{siam}
\bibliography{Refs2point}	

\Addresses
\end{document}